\documentclass[a4paper,12pt]{article}
\usepackage{amsfonts,amsmath,amsthm,amssymb}
\usepackage{bbm}
\usepackage{hyperref,tikz}
\usepackage{color}
\usepackage[shortlabels]{enumitem}
\usepackage[margin=1in]{geometry}

\theoremstyle{plain}
\newtheorem{theorem}{Theorem}[section]

\newtheorem{corollary}[theorem]{Corollary}
\newtheorem{lemma}[theorem]{Lemma}
\newtheorem{proposition}[theorem]{Proposition}
\newtheorem{computation}[theorem]{Computation}
\newtheorem*{claim*}{Claim}
\newtheorem*{problem*}{Problem}
\newtheorem*{conjecture*}{Conjecture}

\theoremstyle{definition}
\newtheorem{definition}[theorem]{Definition}

\newtheorem{example}[theorem]{Example}
\newtheorem{remark}[theorem]{Remark}

\newcommand\al{\alpha}
\newcommand\bt{\beta}

\newcommand\lm{\lambda}
\newcommand\sg{\sigma}

\newcommand\Lm{\Lambda}

\newcommand\cF{\mathcal{F}}
\newcommand\cH{\mathcal{H}}
\newcommand\cJ{\mathcal{J}}
\newcommand\cM{\mathcal{M}}
\newcommand\cAM{\mathcal{AM}}

\newcommand\F{\mathbb{F}}

\newcommand\Q{\mathbb{Q}}

\newcommand{\A}{\mathrm{A}}
\newcommand{\B}{\mathrm{B}}
\newcommand{\C}{\mathrm{C}}

\newcommand\la{\langle}
\newcommand\ra{\rangle}
\newcommand\lla{\langle\!\langle}
\newcommand\rra{\rangle\!\rangle}

\DeclareMathOperator{\Aut}{Aut}

\DeclareMathOperator{\Miy}{\mathrm{Miy}}

\newcommand\ad{\mathrm{ad}}
\newcommand\one{\mathbbm{1}}

\newcommand\qu{\frac{1}{4}}
\renewcommand\th{\frac{1}{32}}

\setlist[enumerate,1]{label={\upshape (\alph*)}}
\setlist[enumerate,2]{label={\upshape (\roman*)}}

\title{Full automorphism groups of the axial algebra for $M_{11}$ and related algebras} 
\author{T.M.~Mudziiri~Shumba\footnote{Sobolev Institute of Mathematics, Novosibirsk, Russia, email: tendshumba@gmail.com}, \\ and 
S.~Shpectorov\footnote{School of Mathematics, University of Birmingham, Edgbaston, Birmingham, B15 2TT, UK, email: s.shpectorov@bham.ac.uk}}

\begin{document}
\maketitle

\begin{abstract}
In this paper, in continuation of \cite{aut}, we compute the full automorphism groups of the $286$-dimensional algebra for $M_{11}$, its subalgebras and other related algebras. This includes, in particular, the $101$-dimensional algebra for $L_2(11)$ and the $76$-dimensional algebra for $A_6$. 

While smaller algebras can be handled by the fully automatic nuanced method from \cite{aut}, the larger algebras, mentioned above, require a hybrid method combining computation with hand-made proofs.
\end{abstract}

\section{Introduction}

Axial algebras introduced in \cite{Axial1,Axial2} are commutative non-associative algebras generated by 
axes, which are (primitive) idempotents satisfying a prescribed fusion law, tabulating properties of the 
adjoint action of the axis. For a discussion of the axial paradigm and known classes of axial algebras see 
the survey \cite{MS1}. In particular, the class of algebras of Monster type is defined by the fusion law 
$\cM(\al,\bt)$ in Table \ref{Monster type law} that generalises in a natural way the fusion law 
$\cM(\qu,\th)$ (see Table \ref{Monster law}) found in the Griess algebra for the Monster sporadic simple 
group. We will refer to the latter as the \emph{Monster fusion law} and we will be focussing in this paper 
on algebras with this specific fusion law. Note that this fusion law is graded and consequently all 
such algebras have substantial automorphism groups. Furthermore, according to \cite{aut}, finite-dimensional 
algebras with the Monster fusion law have finite automorphism groups. This explains why we always relate 
finite groups with our algebras.

A large number of algebras with Monster fusion law was computed using the expansion algorithm 
\cite{axialconstruction,kms,database} and also earlier algorithms \cite{Se,PW}. Many (but not all) of these 
algebras are subalgebras of the Griess algebra. So understanding these algebras is important for the study 
of the Griess algebra, with a connection to the area of vertex operator algebras. 

In \cite{aut} we initiated the study of automorphism groups of axial algebras of Monster type, introducing 
new methods, both theoretic and computational, for achieving this task. The original motivation for this 
project was the observation that the dimensions of the algebras we constructed tend to repeat. One possible 
explanation of this phenomenon is that perhaps some of these algebras are isomorphic even though they 
originate from different groups and axets.\footnote{The term axet introduced in \cite{MS} means a closed set 
of axes, i.e., a set of axes that is invariant under the corresponding Miyamoto group. See the exact 
definitions in Section \ref{preliminaries}.} This is not a fanciful thinking: such situation has already been 
observed in the axial universe. E.g., there are infinitely many $2$-generated axial algebras of Jordan type 
half, up to axet-preserving automorphisms, but only finitely many isomorphism types if we ignore the axet 
structure, i.e., if we view them up to general algebra isomorphisms. 

One approach to this isomorphism question is via computing the full automorphism group of the algebra. 
Indeed, if two axial algebras, coming from groups $G$ and $H$, are isomorphic, then both $G$ and $H$ should 
be subgroups of the full automorphism group of the algebras, which should then be bigger than $G$ and/or $H$. 
Equivalently, we should have additional axes in the two algebras, not contained in the axets we started with.

The largest algebra whose automorphism group we were able to compute in \cite{aut} was the $151$-dimensional 
algebra for $\Aut{S_6}$. In the present paper, we significantly improve on the methods of \cite{aut} which allows us to 
compute the automorphism group of the largest constructed algebra to date, $A_{286}$, which has dimension 
$286$, and has the smallest Mathieu sporadic simple group $M_{11}$ as its Miyamoto group. (So this means our 
methods are now powerful enough to handle axial algebras we can currently construct.) Additionally we 
determine the full automorphism groups of some related algebras, e.g. as the subalgebras of $A_{286}$ 
corresponding to the maximal subgroups of $M_{11}$, including algebra $A_{101}$ for $L_2(11)$ and the  
algebra $A_{76}$ for $M_{10}=A_6.2$, as well as some other naturally related algebras. 

Our results are summarised in Table \ref{results}.
\begin{table}\label{results}
\centering
\begin{tabular}{cccccccc}
\hline
$A$&$G$&$G_0$&axet& dim & shape&new&${\rm Aut}(A)$\\\hline
$A_{17}$&$2{\cdot}S_4$&$S_4$&$1+12^*$&$17$&$4B6A$&yes& $2\times S_4$\\\hline
$A_{36}$&$S_5$&$S_5$&$15+10^*$&$36$&$4B$ &no&$S_5$\\\hline
$A_{24}$&	&$S_3\wr 2$ &$9+(6+6)^*$&$24$&$4B$&yes&$U_3(2){:}2$\\
$A_{18}$&$U_3(2){:}2$& $S_3^2$&$9+(3+3)^*$&$18$&$6A6A$ &no&$S_3\wr 2$\\
$A_{12}$&	&$3^2:2$&$9$&$12$&$3A3A3A3A$ & yes&$AGL(2,3)$\\\hline
$A_{101}$&$L_2(11)$&$L_2(11)$&$55$&$101$&$6A5A$ &no&${\rm Aut}(L_2(11))$\\\hline
$A_{76}$&$M_{10}$&$A_6$&$45$&$76$&$4B3A3A$ &no&${\rm Aut}(A_6)$ \\\hline
$A_{286}$&$M_{11}$&$M_{11}$&$165$&$286$&$4B$&no&$M_{11}$\\\hline
\end{tabular}
\caption{Algebras studied in this paper}
\end{table}
In this table we indicate the notation for the algebra $A$ adapted in this paper and derived from the 
algebra's dimension (luckily all these algebras $A$ have different dimensions); the group $G$ the algebra $A$ 
was constructed from; the Miyamoto group $G_0$ of $A$; the size and composition of the axet; the shape of the 
algebra $A$ and its dimension; whether the algebra is a known one or it has been discovered within this 
project; and finally, the full automorphism group of the algebra. The axet composition shows how the axet 
splits into orbits under the Miyamoto group $G_0$. An asterisk indicates that the particular orbit or orbits 
within the axet already generate the algebra. This means that the same algebra $A$ may be considered with a 
smaller axet, only including the starred orbits, although the shape of the algebra will then change.

For several of these algebras $A$ the full automorphism group turned out to be bigger than the expected 
group, i.e., the stabiliser of $A$ in $M_{11}$. While the algebra $A_{17}$ was constructed from the maximal 
subgroup $G\cong 2\cdot S_4$ of $M_{11}$, the central involution $z$ of $G$ acts on $A_{17}$ as identity, 
hence the Miyamoto group $G_0$ is in this case the factor group $S_4$ of $G$. However, the axis corresponding 
to $z$ is contained in $A_{17}$ and it is a Jordan axis there, which leads to an additional involutive 
automorphism of $A_{17}$ appearing as the direct factor in $\Aut(A_{17})$. The algebra $A_{24}$ is also the 
subalgebra of $A_{286}$ corresponding to the maximal subgroup $U_3(2):2$ of $M_{11}$. This algebra has no new 
automorphisms. However, its subalgebras, $A_{18}$ and $A_{12}$, constructed from natural subaxets of the axet 
of $A_{24}$, have extra symmetries. The full automorphism group of $A_{18}$ is twice as big as what we can 
find in $M_{11}$, and $\Aut(A_{12})$ is bigger than its stabiliser in $M_{11}$ by a factor of $24$. These are 
very interesting algebras and the way they are embedded in $A_{24}$ suggests that there may be a way to 
construct even larger algebras using $A_{18}$ and $A_{12}$ as building blocks. Also, for the algebras 
$A_{101}$ and $A_{76}$, constructed from the maximal subgroups $L_2(11)$ and $M_{10}$, respectively, the 
final full automorphism groups are larger by a factor of $2$.

Let us also mention that the expansion algorithm currently fails to complete the shapes of the algebras 
$A_{36}$ and $A_{24}$. We now know that these two cases are non-empty and the complexity of these cases may 
indicate that either there is a second algebra in each of these cases or possibly a larger cover of the known 
algebra.

We also found interesting additional idempotents in some of these algebras. For example, $A_{76}$ is 
generated by $30$ primitive axes of almost Monster type $\mathcal{AM}\left(\frac{1}{2},\frac{1}{8}\right)$. 
This fusion law is just like the Monster type fusion law (see Table \ref{Monster type law}), except 
$\al\ast\al=\{1,0,\al\}$. Under this fusion law, the Miyamoto group of $A_{76}$ is $S_6$.

We organise the paper as follows. In Section \ref{preliminaries}, we give the background information about 
axial algebras encompassing the important definitions and results. In particular, the concept of shape is 
explained in detail and several examples are provided. The computational functionality developed in the study 
of automorphism groups of axial algebras is also reviewed in this section. In Section \ref{small groups}, we 
discuss algebras admitting the smaller maximal subgroups of $M_{11}$, namely $2{.}S_4$, $S_5$ and 
$U_3(2){:}2$. These algebras were investigated by purely computational tools. The remainder of the paper
looks at the larger algebras for the maximal subgroups $M_{10}$ and $L_2(11)$, as well as $M_{11}$ itself.
Specifically, in Section \ref{M10 alg}, we complete the algebra $A_{76}$ which admits $M_{10}$ and has 
Miyamoto group $A_6$. The algebra $A_{101}$ for $L_2(11)$ is studied in Section \ref{L211 algebra}. Finally, 
we determine the full automorphism group of the largest algebra $A_{286}$ for $M_{11}$ in Section \ref{M11 
algebra}.  

All the computations, save one, were carried out on a laptop with 16GB of RAM having AMD Ryzen processor and
running Arch Linux. These computations were carried out using the set of MAGMA \cite{MAGMA} routines, as 
described in Subsection \ref{package}. The routines and the log files of our calculations are available on 
GitHub \cite{m11 aut}. The longest computation (finding all idempotents of a given length in an
$18$-dimensional subalgebra) was first carried out by E.A. O'Brien and then also independently verified by us 
in a University of Birmingham Unix server run by D.A. Craven. 

\medskip
{\bf Acknowledgement:} The authors gratefully acknowledge partial support of the Mathematical Center in 
Akademgorodok under the agreements No. 075-15-2019-1675 and 075-15-2022-281 with the Ministry of 
Science and Higher Education of the Russian Federation.

\section{Preliminaries} \label{preliminaries}

In this section we present background information about the axial algebras
in general and the specific examples we study in this paper. We also review 
the computational tools created for the automorphism project \cite{aut}.
Throughout, we follow the notation for finite groups, as found for example 
in \cite{atlas}. 

\subsection{Axial algebras}

The algebras we are studying in this paper were discovered within the paradigm 
of axial algebras, so we start by discussing the basics of it. For a more 
detailed discussion see e.g. \cite{MS1}.

\begin{definition}
A \emph{fusion law} is a pair $(\cF,\ast)$, where $\cF$ is a set and $\ast$ 
a binary operation on $\cF$ taking values in the power set of $\cF$.
\end{definition}

The particular fusion law of relevance for this paper is shown in Table 
\ref{Monster law}.
\begin{table}
    \centering
    \renewcommand{\arraystretch}{1.7}
	\begin{tabular}{|c||c|c|c|c|}
		\hline
		$\ast$ & $1$ & $0$ & $\qu$ & $\th$\\
		\hline\hline
		$1$ & $1$ & & $\qu$ & $\th$\\
		\hline
		$0$ & & $0$ & $\qu$ & $\th$\\
		\hline
		$\qu$ & $\qu$ & $\qu$ & $1,0$ & $\th$\\
		\hline
		$\th$ & $\th$ & $\th$ & $\th$ & $1,0,\qu$\\
		\hline
	\end{tabular}
	\caption{Monster fusion law $\cM=\cM(\qu,\th)$}
	\label{Monster law}
\end{table}
Here $\cF=\cM=\{1,0,\qu,\th\}$ and the table cell corresponding to $\lm$ and 
$\mu$ lists the elements of $\lm\ast\mu$. For example, $1\ast 1=\{1\}$, $0\ast 
1=\emptyset$, $\qu\ast\qu=\{1,0\}$, and so on.

\bigskip
Let $A$ be a commutative non-associative algebra over the field $\mathbb{F}$. 
For an element $a\in A$, we write $\ad_a$ for the adjoint map $A\to A$ given 
by $u\mapsto au$. For $\lm\in\F$, the $\lm$-eigenspace of $\ad_a$ will be 
denoted by $A_\lm(a)$. That is,
$$A_\lm(a)=\{u\in A\mid au=\lm u\}.$$
For $\Lm\subseteq\F$, 
$$A_\Lm(a)=\oplus_{\lm\in\Lm}A_\lm(a).$$

Let now $(\cF,\ast)$ be a fusion law, with $\cF\subseteq\F$.

\begin{definition}
A nonzero idempotent $a\in A$ is an ($\cF$-)axis if
\begin{itemize}[(a)]
\item $A=A_\cF(a)$; and
\item $A_\lm(a)A_\mu(a)\subseteq A_{\lm\ast\mu}(a)$ for all $\lm,\mu\in\cF$.
\end{itemize}
\end{definition} 

Condition (a) means that $\ad_a$ is diagonalisable with spectrum fully in 
$\cF$, while condition (b) means that the fusion law limits the values of the algebra product 
on eigenvectors of $\ad_a$.

It is easy to see that $a\in A_1(a)$, so $1\in\cF$.

\begin{definition}
An axis $a\in A$ is \emph{primitive} if $A_1(a)=\F a$, that is, $A_1(a)$ is $1$-dimensional spanned by $a$.
\end{definition}

Now we can give the key definition.

\begin{definition}
The algebra $A$ is a (primitive) ($\cF$-)axial algebra if it is generated by a set of (primitive) axes.  
\end{definition}

The following property of the fusion law has important implications for the structure of the algebra.

\begin{definition}
If $\cF$ is a fusion law and $0\in \cF$, we say that $\cF$ is \emph{Seress} if for all $\lm \in \cF$, 
$0\ast\lm\subseteq \{\lm\}$, i.e., $0$ behaves the same way as $1$ in $\cF$.
\end{definition}

When the fusion law $\cF$ is Seress then, first of all, the eigenspace $A_0(a)$ is a subalgebra of $A$ for 
each axis $a$. Furthermore, $A_{\{1,0\}}(a)=A_1(a)\oplus A_0(a)$ is also a subalgebra. Secondly, this property 
assures that each eigenspace $A_\lm(a)$ is a module for both of these algebras, i.e., 
$A_{\{1,0\}}(a)A_\lm(a)\subseteq A_\lm(a)$. There is also partial associativity in $A$ implied by the Seress 
property, but this is not going to be important for this paper.

\medskip
Often an axial algebra admits a natural bilinear form.

\begin{definition}
Suppose that $A$ is an axial algebra. A nonzero bilinear form on $A$ is called a \emph{Frobenius form} provided that
$$(uv,w)=(u,vw)$$
for all $u,v,w\in A$. 
\end{definition}

Note that a nonzero multiple of a Frobenius form is again Frobenius. Typically, a Frobenius form is scaled so 
that axes have (square) length $1$, when this is possible.

\bigskip
It can be seen from the above, that a fusion law can be used to define a class 
of axial algebras. In particular, the class we are interested in is the class 
of (primitive) algebras of Monster type, which is defined by the Monster type 
fusion law $\cM(\al,\bt)$ in Table \ref{Monster type law},
\begin{table}
    \centering
    \renewcommand{\arraystretch}{1.7}
	\begin{tabular}{|c||c|c|c|c|}
		\hline
		$\ast$ & $1$ & $0$ & $\al$ & $\bt$\\
		\hline\hline
		$1$ & $1$ & & $\al$ & $\bt$\\
		\hline
		$0$ & & $0$ & $\al$ & $\bt$\\
		\hline
		$\al$ & $\al$ & $\al$ & $1,0$ & $\bt$\\
		\hline
		$\bt$ & $\bt$ & $\bt$ & $\bt$ & $1,0,\al$\\
		\hline
	\end{tabular}
	\caption{Monster type fusion law $\cM=\cM(\al,\bt)$}
	\label{Monster type law}
\end{table}
looking like the Monster fusion law in Table \ref{Monster 
law}, but with arbitrary distinct $\al,\bt\in\F\setminus\{1,0\}$ instead of 
$\qu$ and $\th$, respectively. The class of algebras of Monster type appeared first in \cite{RT,r} and the motivating example was the Griess algebra 
for the Monster sporadic simple group, which arises exactly for 
$(\al,\bt)=\left(\qu,\th\right)$. This explains the interest to this specific 
fusion law.

This case of algebras with the Monster fusion law is probably the most investigated one. 
In particular, nearly two hundred of such algebras, corresponding to various 
small finite groups, were constructed using the expansion algorithm 
\cite{axialconstruction}. (See also \cite{Se} and \cite{PW}.) A few of these algebras will be the focus of this paper.

\medskip
In studying subalgebras of axial algebras, a class of fusion laws which are essentially of Monster type has 
been encountered. Namely, a fusion law $\cF$ is of \emph{almost Monster} type $(\al, \bt)$, denoted $\cAM(\al, \bt)$, if $\cF$ differs from the Monster fusion law $\mathcal{M}(\al,\bt)$ only in that $\al\ast\al=\{1,0,\al\}$ (see Table \ref{almost Monster type}).
\begin{table}
\centering
\renewcommand{\arraystretch}{1.7}
\begin{tabular}{|c||c|c|c|c|}
\hline
$\ast$ & $1$ & $0$ & $\al$ & $\bt$\\
\hline\hline
$1$ & $1$ & & $\al$ & $\bt$\\
\hline
$0$ & & $0$ & $\al$ & $\bt$\\
\hline
$\al$ & $\al$ & $\al$ & $1,0,\al$ & $\bt$\\
\hline
$\bt$ & $\bt$ & $\bt$ & $\bt$ & $1,0,\al$\\
\hline
\end{tabular}
\caption{Almost Monster type fusion law $\cAM(\al,\bt)$}
\label{almost Monster type}
\end{table}

\subsection{Miyamoto group}

Next we define graded fusion laws following the approach of \cite{dpsv}. Given a group $T$, we can introduce the
\emph{group fusion law} $(T,\ast)$, where $t\ast s=\{ts\}$ for all $t,s\in T$. 
A \emph{morphism} of fusion laws $(\cF,\ast)$ and $(\cH,\ast)$ is a map $\phi:\cF\to\cH$ such that
$$\phi(\lm\ast\mu)\subseteq\phi(\lm)\ast\phi(\mu),$$
for all $\lm,\mu\in\cF$.

We are now prepared for the main concept.

\begin{definition}
Suppose that $\cF$ is a fusion law and $T$ is a group. A 
$T$-grading of $\cF$ is a morphism from $\cF$ to the group fusion 
law of $T$. 
\end{definition} 

For our purposes, $T$ will always be abelian and small. The grading $\phi$ is called \emph{adequate} if $T=\la\phi(\cF)\ra$.
This condition assures that the grading is by $T$ itself and not by a proper subgroup of $T$.
\begin{table}[!ht]
\setlength{\tabcolsep}{4pt}
\renewcommand{\arraystretch}{1.5}
%\addtolength\tabcolsep{5pt}
\centering
\footnotesize
\begin{tabular}{c|c|c}
\hline
Type & Basis & Products \& form \\ \hline
$2\textrm{A}$ & \begin{tabular}[t]{c} $a_0$, $a_1$, \\ $a_\rho$ 
\end{tabular} & 
\begin{tabular}[t]{c}
$a_0 \cdot a_1 = \frac{1}{8}(a_0 + a_1 - a_\rho)$ \\
$a_0 \cdot a_\rho = \frac{1}{8}(a_0 + a_\rho - a_1)$ \\
$(a_0, a_1) = (a_0, a_\rho)= (a_1, a_\rho) = \frac{1}{8}$
\vspace{4pt}
\end{tabular}
\\
$2\textrm{B}$ & $a_0$, $a_1$ &
\begin{tabular}[t]{c}
$a_0 \cdot a_1 = 0$ \\
$(a_0, a_1) = 0$
\vspace{4pt}
\end{tabular}
\\
$3\textrm{A}$ & \begin{tabular}[t]{c} $a_{-1}$, $a_0$, \\ $a_1$, $u_\rho$ \end{tabular} &
\begin{tabular}[t]{c}
$a_0 \cdot a_1 = \frac{1}{2^5}(2a_0 + 2a_1 + a_{-1}) - \frac{3^3\cdot5}{2^{11}} u_\rho$ \\
$a_0 \cdot u_\rho = \frac{1}{3^2}(2a_0 - a_1 - a_{-1}) + \frac{5}{2^{5}} u_\rho$ \\
$u_\rho \cdot u_\rho = u_\rho$, $(a_0, a_1) = \frac{13}{2^8}$ \\
$(a_0, u_\rho) = \frac{1}{4}$, $(u_\rho, u_\rho) = \frac{2^3}{5}$ 
\vspace{4pt}
\end{tabular}
\\
$3\textrm{C}$ & \begin{tabular}[t]{c} $a_{-1}$, $a_0$, \\ $a_1$ \end{tabular} &
\begin{tabular}[t]{c}
$a_0 \cdot a_1 = \frac{1}{2^6}(a_0 + a_1 - a_{-1})$ \\
$(a_0, a_1) = \frac{1}{2^6}$
\vspace{4pt}
\end{tabular}
\\
$4\textrm{A}$ & \begin{tabular}[t]{c} $a_{-1}$, $a_0$, \\ $a_1$, $a_2$, \\ $v_\rho$ \end{tabular} &
\begin{tabular}[t]{c}
$a_0 \cdot a_1 = \frac{1}{2^6}(3a_0 + 3a_1 + a_{-1} + a_2 - 3v_\rho)$ \\
$a_0 \cdot v_\rho = \frac{1}{2^4}(5a_0 - 2a_1 - a_2 - 2a_{-1} + 3v_\rho)$ \\
$a_0 \cdot a_2 = 0$, $v_\rho \cdot v_\rho = v_\rho$ \\
$(a_0, a_1) = \frac{1}{2^5}$, $(a_0, a_2) = 0$\\
$(a_0, v_\rho) = \frac{3}{2^3}$, $(v_\rho, v_\rho) = 2$
\vspace{4pt}
\end{tabular}
\\
$4\textrm{B}$ & \begin{tabular}[t]{c} $a_{-1}$, $a_0$, \\ $a_1$, $a_2$, \\ $a_{\rho^2}$ \end{tabular} &
\begin{tabular}[t]{c}
$a_0 \cdot a_1 = \frac{1}{2^6}(a_0 + a_1 - a_{-1} - a_2 + a_{\rho^2})$ \\
$a_0 \cdot a_2 = \frac{1}{2^3}(a_0 + a_2 - a_{\rho^2})$ \\
$a_0 \cdot a_{\rho^2} = \frac{1}{2^3}(a_0 + a_{\rho^2} - a_2)$, $a_{\rho^2} \cdot a_{\rho^2} = a_{\rho^2}$ \\
$(a_0, a_1) = \frac{1}{2^6}$, $(a_0, a_2) = (a_0, a_{\rho^2})= \frac{1}{2^3}$, $(a_{\rho^2}, a_{\rho^2}) = 1$
\vspace{4pt}
\end{tabular}
\\
$5\textrm{A}$ & \begin{tabular}[t]{c} $a_{-2}$, $a_{-1}$,\\ $a_0$, $a_1$,\\ $a_2$, $w_\rho$ \end{tabular} &
\begin{tabular}[t]{c}
$a_0 \cdot a_1 = \frac{1}{2^7}(3a_0 + 3a_1 - a_2 - a_{-1} - a_{-2}) + w_\rho$ \\
$a_0 \cdot a_2 = \frac{1}{2^7}(3a_0 + 3a_2 - a_1 - a_{-1} - a_{-2}) - w_\rho$ \\
$a_0 \cdot w_\rho = \frac{7}{2^{12}}(a_1 + a_{-1} - a_2 - a_{-2}) + \frac{7}{2^5}w_\rho$ \\
$w_\rho \cdot w_\rho = \frac{5^2\cdot7}{2^{19}}(a_{-2} + a_{-1} + a_0 + a_1 + a_2)$ \\
$(a_0, a_1) = (a_0, a_2)= \frac{3}{2^7}$, $(a_0, w_\rho) = 0$, $(w_\rho, w_\rho) = \frac{5^3\cdot7}{2^{19}}$
\vspace{4pt}
\end{tabular}
\\
$6\textrm{A}$ & \begin{tabular}[t]{c} $a_{-2}$, $a_{-1}$,\\ $a_0$, $a_1$,\\ $a_2$, $a_3$, \\ $a_{\rho^3}$, $u_{\rho^2}$ \end{tabular} &
\begin{tabular}[t]{c}
$a_0 \cdot a_1 = \frac{1}{2^6}(a_0 + a_1 - a_{-2} - a_{-1} - a_2 - a_3 + a_{\rho^3}) + \frac{3^2\cdot5}{2^{11}}u_{\rho^2}$ \\
$a_0 \cdot a_2 = \frac{1}{2^5}(2a_0 + 2a_2 + a_{-2}) - \frac{3^3\cdot5}{2^{11}}u_{\rho^2}$ \\
$a_0 \cdot u_{\rho^2} = \frac{1}{3^2}(2a_0 - a_2 - a_{-2}) + \frac{5}{2^5}u_{\rho^2}$, $u_{\rho^2} \cdot u_{\rho^2} = u_{\rho^2}$ \\
$a_0 \cdot a_3 = \frac{1}{2^3}(a_0 + a_3 - a_{\rho^3})$, $a_{\rho^3} \cdot a_{\rho^3} = a_{\rho^3}$, $a_{\rho^3} \cdot u_{\rho^2} = 0$ \\
$(a_0, a_1) = \frac{5}{2^8}$, $(a_0, a_2) = \frac{13}{2^8}$, $(a_0, u_{\rho^2}) = \frac{1}{4}$, $(u_{\rho^2}, u_{\rho^2}) = \frac{2^3}{5}$ \\
$(a_0, a_3) = (a_0, a_{\rho^3}) = \frac{1}{8}$,  $(a_{\rho^3}, a_{\rho^3}) = 1$, $(a_{\rho^3}, u_{\rho^2}) = 0$
\end{tabular}\\
\hline
\end{tabular}
\caption{Norton-Sakuma algebras}\label{Norton-Sakuma}
\end{table}

The fusion laws $\cM(\al,\bt)$ and $\cAM(\al,\bt)$ are graded by the group of order $2$, $T=C_2=\{\pm 1\}$. 
Namely, the grading map sends $1$, $0$, and $\al$ to $1$ and it sends $\bt$ to $-1$. Clearly, this grading is 
adequate, since it is surjective.

\medskip
From now on we consider a fusion law $\cF$ graded by an abelian group $T$ via a grading $\phi$. Let 
$T^\ast$ be the set of (linear) characters of $T$, that is the set of all homomorphisms from $T$ to the 
multiplicative group of the field $\cF$.

\begin{definition}
Suppose that $(A,X)$ is an $\cF$-axial algebra, $a\in X$ an axis, and $\chi\in T^\ast$. The \emph{Miyamoto 
automorphism} $\tau_a(\chi)$ of $A$ corresponding to $a$ and $\chi$ is the linear map $A\to A$ that acts as 
the scalar $\chi(\phi(\lm))$ on $A_\lm(a)$ for each $\lm\in\cF$.
\end{definition}

It is easy to see that $\tau_a(\chi)$ is indeed an automorphism of $A$. Furthermore, the map $\tau_a$ for a 
fixed axis $a$ is a homomorphism from $T^\ast$ to the automorphism group of $A$, $\Aut(A)$. The image of 
$\tau_a$ is called the \emph{axial subgroup} corresponding to $a$.

\begin{definition}
The \emph{Miyamoto group} $\Miy(A)$ of $A$ is the subgroup of $\Aut(A)$ generated by the axial subgroups 
for all generating axes $a\in X$. In other words,
$$\Miy(A)=\la\tau_a(\chi)\mid a\in X,\chi\in T^\ast\ra.$$
\end{definition}

Since we are focussing in this paper on fusion laws graded by the group $T=\{\pm 1\}$ of order $2$, we can 
simply our notation for this specific case. Note that in this case we must assume that $\F$ is of characteristic 
not $2$, or else the grading does not work. (In fact, we will work in characteristic zero later on.) Hence the 
character group $T^\ast$ contains the unique nontrivial character giving us the unique non-identity Miyamoto automorphism for each given axis $a$. We will denote this unique non-identity automorphism
by $\tau_a$ and call it the \emph{Miyamoto involution} corresponding to $a$. Indeed, it is easy to see that
$\tau_a$ squares to the identity, and so it is normally an involution.

Note that the involutions $\tau_a$ do not need to be distinct for different axes $a$.

\begin{definition}
Axes $a$ and $b$, $a\neq b$, are called \emph{twins} if $\tau_a=\tau_b$.
\end{definition}

While $\tau_a$ is usually an involution, it can also be the identity when $A_\bt(a)=0$ (or, more generally, all 
eigenspaces $A_\lm(a)$ with $\phi(\lm)=-1$ are trivial). 

\begin{definition}
An axis $a$ of Monster type $(\al,\bt)$ is said to be a \emph{Jordan axis} if $A_\bt(a)=0$.  
\end{definition}

The reason for the term is that $a$ then satisfies a smaller Jordan type law (see Table \ref{Jordan type}). In this case $\tau_a$ is the identity, as we have already mentioned. However, the Jordan type law admits its own grading by the group of order $2$. Namely, the grading sends $1$ and $0$ to $1$ and it sends $\al$ to $-1$. 
This leads to an extra involutive automorphism $\sg_a$ that may or may not be contained in the Miyamoto 
group of $A$.
\begin{table}
\centering
\renewcommand{\arraystretch}{1.7}
\begin{tabular}{|c||c|c|c|}
\hline
$\ast$ & $1$ & $0$ & $\al$\\
\hline\hline
$1$ & $1$ & & $\al$\\
\hline
$0$ & & $0$ & $\al$\\
\hline
$\al$ & $\al$ & $\al$ & $1,0$\\
\hline
\end{tabular}
\caption{Jordan type fusion law $\cJ(\al)$}
\label{Jordan type}
\end{table}

We will encounter twins and Jordan axes in the algebras below.

\medskip
For an automorphism $\phi$ of $A$, if $a$ is an axis then, clearly, $\phi(a)$ is also an axis. 

\begin{definition}
Suppose that $(A,X)$ is an axial algebra with a graded fusion law. The \emph{closure} $\bar X$ of the set of generating axes $X$ is defined as 
$$\bar X=\{\phi(a)\mid a\in X,\phi\in\Miy(A)\}.$$
The set $X$ is said to be \emph{closed} if $\bar X=X$.
\end{definition}

Note that $\bar X$ also consists of axes that generate $A$, so we can now 
view $A$ as an axial algebra with the larger generating set. It was shown 
in \cite{axialconstruction} that the Miyamoto group of $A$ is not affected 
by this change of the set of generators, and consequently, the closure of 
$\bar X$ is again $\bar X$, that is, $\bar X$ is closed. The advantage of 
the closed set of generators is that the Miyamoto group acts on it and this 
action is faithful. In what follows we will typically assume that our set $X$ 
is closed.

According to \cite{MS}, an \emph{axet} is a set with an action of a group on it 
together with a tau map from the set to the group, satisfying some natural axioms 
involving also a graded fusion law. Every closed set of axes in an algebra with a graded 
fusion law forms an axet. We do not need the precise definitions, but it is convenient 
to use the short term `axet'  for closed sets of axes, as we did in the 
introduction and will continue doing throughout the paper.

Finally, there is one more concept that appears in the description of algebras below. 
An axial algebra $(A,X)$, with closed $X$, is said to be \emph{$k$-closed} if 
$A$ is spanned by the products of axes from $X$ with at most $k$ factors. 
E.g., an algebra is $1$-closed if $X$ spans $A$.

\subsection{Shapes}

It was proved in \cite{generic} that every $2$-generated (i.e., with $|X|=2$, where $X$ is not necessarily closed) algebra of Monster type is isomorphic
to one of the Norton-Sakuma algebras, first found by Norton as subalgebras of the Griess algebra.
Table \ref{Norton-Sakuma} contains basic information about these algebras, where $a_0$ and $a_1$ are the two generating axes.
Note that all Norton-Sakuma algebras admit a Frobenius form which is scaled so that all axes have length $1$.
Hence Table \ref{Norton-Sakuma} also contains the values of this Frobenius form. Column 1 of the table contains
the names of the algebras which are derived from the names of the conjugacy classes of the Monster group
corresponding in a natural way to the Norton-Sakuma algebras.

Norton-Sakuma algebras $B$ are unital, that is, they all contain an identity element $\one_B$. In Table \ref{length of one},
\begin{table}[!ht]
\centering
\renewcommand{\arraystretch}{2}
\begin{tabular}{ccccccccc}\hline
Type of $B$&$2\A$&$2\B$&$3\A$&$3\C$&$4\A$&$4\B$&$5\A$&$6\A$\\\hline
$(\one_B,\one_B)$&$\frac{12}{5}$&$2$&$\frac{116}{35}$&$\frac{32}{11}$&$4$&$\frac{19}{5}$&$\frac{32}{7}$&$\frac{51}{10}$\\\hline 
\end{tabular}
\caption{Identity length} \label{length of one}
\end{table}
we provide information concerning the length of $\one_B$ in 
Norton-Sakuma algebras. This is a convenient parameter 
distinguishing these algebras.

\medskip
Related to the known list of Norton-Sakuma algebras, we have the important 
concept of the shape of an algebra of Monster type. This was first introduced 
by Ivanov (e.g., see \cite{ipss}) in the context of Majorana algebras and then 
generalised to all axial algebras. We will give this definition in the context 
of algebras of Monster type.

\medskip
Suppose that $(A,X)$ is an axial algebra of Monster type and $X$ is closed. 
Let $\binom{X}{2}$ be the set of all unordered pairs of axes from $X$. We 
will view $\binom{X}{2}$ as a partially ordered set, where we have 
$(a,b)\leq(c,d)$ if and only if $a$ and $b$ are contained in the closure 
$\overline{\{c,d\}}=c^{\Miy(B)}\cup d^{\Miy(B)}$ of $\{c,d\}$. Here 
$B=\lla c,d\rra$. In particular, this implies that $\lla a,b\rra\subseteq
\lla c,d\rra$. 

\begin{definition}
The \emph{shape} of $A$ is the map attaching to each $(a,b)\in\binom{X}{2}$ 
the type of the subalgebra $\lla a,b\rra$ of $A$.
\end{definition}

The shape is a convenient invariant distinguishing different algebras of 
Monster type with the same Miyamoto group. Note that the first symbol (the 
number) in the Norton-Sakuma type of $\lla a,b\rra$ always coincides with 
the size of $\overline{\{a,b\}}$. Hence the only choices we have in the shape 
are the letters in the types, where they are not uniquely identified by the 
corresponding number. That is, $2A$ or $2B$, $3A$ or $3C$, and $4A$ or $4B$.

Recall that we view $\binom{X}{2}$ as a partially ordered set. If $(a,b)\leq 
(c,d)$ then, as we have already pointed out, $\lla a,b\rra$ is a subalgebra 
of $\lla c,d\rra$, which means that the type of $\lla c,d\rra$ determines 
the type of $\lla a,b\rra$. So the shape must be consistent in this sense 
with respect to the partial order on $\binom{X}{2}$.

In Table \ref{inclusions}, we list all possible inclusions between two 
Norton-Sakuma algebras.
\begin{table}[ht!]
\centering
\begin{tabular}{|c|c|}
\hline
Algebra & Subalgebras\\
\hline\hline
4A & 2B\\
4B & 2A\\
6A & 2A,3A\\
\hline
\end{tabular}
\caption{Inclusions among Norton-Sakuma algebras}
\label{inclusions}
\end{table}
We see from this table that dependences between the types of $2$-generated 
subalgebras work both ways. Namely, if $(a,b)\leq(c,d)$ then not only the 
type of $\lla c,d\rra$ determines the type of $\lla a,b\rra$, but also the 
type of $\lla a,b\rra$ determines the type of $\lla c,d\rra$. 

Of course, the set $\binom{X}{2}$ can be huge and so the shape contains a 
lot of repeated information. Namely, if $(a,b)$ and $(c,d)$ from 
$\binom{X}{2}$ are conjugate by an element of $\Miy(A)$ then $\lla a,b\rra$ 
and $\lla c,d\rra$ are conjugate and hence isomorphic. This means that we 
can fold the shape over the action of $\Miy(A)$ without losing any 
information.

In a text, we represent the (folded) shape simply as a sequence of 
Norton-Sakuma types from the shape. While not a perfect representation, 
this serves the purpose of distinguishing algebras and it does not lead 
to ambiguity in most cases. Furthermore, dependencies between the values 
of the shape map allow us to shorten the shape representation even further 
by indicating only the essential type choices. Below a shape is recorded 
as a bullet point followed by the list of essential choices of the 
Norton-Sakuma types. We introduce the bullet point in the beginning of 
the shape because we sometimes have situations with no choice at all, 
and this will be represented by just the bullet point.

One final remark about the shapes is that, when the map $a\mapsto\tau_a$ 
is injective on every $\Miy(A)$-orbit on $X$, the size of 
$\overline{\{a,b\}}$ coincides with the order $|\tau_a\tau_b|$, and this 
is an easier way to identify the number in the type of $\lla a,b\rra$. 
Note that this orbit injectivity condition is always satisfied in the 
algebras we investigate in this paper.  

\medskip
We will illustrate the concept of shape with the following examples.

\begin{example} \label{286 shape}
Let $A=A_{286}$ be the $286$-dimensional algebra with Monster fusion law for 
the sporadic simple group $G=M_{11}$. This algebra was first constructed 
by Seress \cite{Se} and then also confirmed using the expansion algorithm. 
The closed set of axes $X$ in this algebra consists of $165$ axes. Note 
that $G$ contains exactly $165$ involutions, 
forming a single conjugacy class $C$, and so it is no surprise that the 
map $a\mapsto\tau_a$ is a natural bijection between $X$ 
and $C$. This means that we can identify $X$ with $C$ and hence derive 
the shape of $A$ looking at the action of $G$ on $C$.

The group $G$ acts on $\binom{C}{2}$ of size $13\,530$ with six orbits of lengths 
$660$, $990$, $1980$, $1980$, $3960$ and $3960$. Hence the folded version 
of $\binom{C}{2}$ has only six nodes, and it is shown in Figure \ref{M11 
shape}.
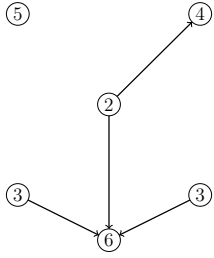
\begin{figure}[!htb]
\centering
\scalebox{0.6}{
\begin{tikzpicture}
\draw (0,0) circle (0.25cm);
\node (0,0) at (0,0){$2$};
\draw (-2,2) circle (0.25cm);
\draw (2,2) circle (0.25cm);
\draw (2,-2) circle (0.25cm);
\draw (-2,-2) circle (0.25cm);
\draw (0,-3) circle (0.25cm);
\node at (0,-3){$6$};
\node at (-2,2){$5$};
\node at (-2,-2){$3$};
\node at (2,2){$4$};
\node at (2,-2){$3$};
\draw[thick, ->] (0,-0.25)--(0,-2.75);
\draw [thick, ->] ({sqrt(2)/8}, {sqrt(2)/8})--({2-sqrt(2)/8},{2-sqrt(2)/8});
\draw [thick, ->] ({-2+sqrt(5)/10},{-2-sqrt(5)/20})--({-sqrt(5)/10},{-3+sqrt(5)/20});
\draw[thick,->] ({2-sqrt(5)/10},{-2-sqrt(5)/20})--({sqrt(5)/10}, {-3+sqrt(5)/20});
\end{tikzpicture}
}
\caption{The folded shape diagram of $A_{286}$}
\label{M11 shape}
\end{figure}
Each node in this figure contains the order $|ab|$, where the pair 
$(a,b)\in\binom{C}{2}$ represents the orbit corresponding to the node. This is the number from the 
Norton-Sakuma type for $(a,b)$.

Looking at the dependencies between the node types, since the bottom node 
must clearly be of Norton-Sakuma type $6A$, the adjacent nodes must be of 
types $3A$, $2A$, and $3A$. (See Table \ref{inclusions}.) Furthermore, 
the node above and adjacent to $2A$ must be $4B$. Finally, the isolated 
node is clearly $5A$. So we do not have any choices, and the algebra $A$
has the unique empty shape $\bullet$.
\end{example}

The further two examples deal with subalgebras of the $M_{11}$ algebra.

\begin{example}
The first subalgebra we consider here corresponds to a subgroup $H\cong 
L_2(11)$ of $G=M_{11}$. This subalgebra $A_{101}$, of dimension $101$, 
is generated by $55$ axes corresponding to the involutions contained in 
$H$. Clearly, $H$ is the Miyamoto group of $A_{101}$. It has six orbits 
on the set of pairs of involutions, with the folded partially ordered set 
shown in Figure \ref{L211 shape}.
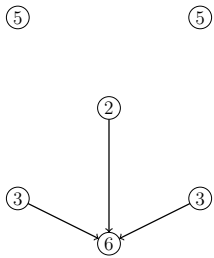
\begin{figure}[!h]
\centering
\scalebox{0.6}{
\begin{tikzpicture}
\draw (0,0) circle (0.25cm);
\node (0,0) at (0,0){$2$};
\draw (-2,2) circle (0.25cm);
\draw (2,2) circle (0.25cm);
\draw (2,-2) circle (0.25cm);
\draw (-2,-2) circle (0.25cm);
\draw (0,-3) circle (0.25cm);
\node at (0,-3){$6$};
\node at (-2,2){$5$};
\node at (-2,-2){$3$};
\node at (2,2){$5$};
\node at (2,-2){$3$};
\draw[thick, ->] (0,-0.25)--(0,-2.75);
\draw [thick, ->] ({-2+sqrt(5)/10},{-2-sqrt(5)/20})--({-sqrt(5)/10},{-3+sqrt(5)/20});
\draw[thick,->] ({2-sqrt(5)/10},{-2-sqrt(5)/20})--({sqrt(5)/10}, {-3+sqrt(5)/20});
\end{tikzpicture}
}
\caption{The folded shape diagram of $A_{101}$}
\label{L211 shape}
\end{figure}
Here again the shape $\bullet$ is unique, since all nodes, where we can have some 
choice, are adjacent to the bottom node carrying the number $6$. So it must be $6A$, 
forcing the adjacent nodes to be $3A$, $2A$, and $3A$, respectively. The two isolated 
nodes are both $5A$.
\end{example}

Note that the shape in this second example is unique regardless of the 
embedding of $A_{101}$ into $A_{286}$. The situation is different for our 
final example. 

\begin{example}
Here we consider the subalgebra $A_{76}$ corresponding to a subgroup 
$K=M_{10}$ of $G=M_{11}$. This $76$-dimensional subalgebra is generated by 
the $45$ axes corresponding to the involutions from $K$. Note that the 
Miyamoto group of $A_{76}$ is not $K$, but rather an index $2$ subgroup 
of $K$ isomorphic to the alternating group on six letters. This is because 
$K$ is not generated by its involutions.

The group $K$ has seven orbits in its action on the unordered pairs from its
set of $45$ involutions, as shown in Figure \ref{A76shape}. The orbits bearing
$2$ are both adjacent to the orbit carrying $4$, so if the Norton-Sakuma type
of the $4X$ orbit is $4A$, then the $2X$ orbits are both $2B$, and likewise,
we have both $2A$ if the Norton-Sakuma type of $4X$ is $4B$. The orbits with the
number $3$ can be $3A$ or $3C$, and {\it a priori}, we have no further information to
distinguish between the two options. The same applies for the choice $4A$ versus $4B$. 
Thus, without the knowledge of the embedding of $A_{76}$ in $A_{286}$, we have eight
possibilities for the shape, namely, $\bullet 3A3A4A$, $\bullet 3A3A4B$,
$\bullet 3A3C4A$, $\bullet 3A3C4B$, $\bullet 3C3A4A$, $\bullet 3C3A4B$,
$\bullet 3C3C4A$ and $\bullet 3C3C4B$. 
\begin{figure}[!h]
\centering
\begin{tikzpicture}
\draw  (0,0) circle (0.25cm) node (0,0){$4$};
\draw (-2,2) circle (0.25cm) node (-2,2){$2$};
\draw (2,2) circle (0.25cm) node (2,2){$2$};
\draw (-2,-2) circle (0.25cm) node (-2,-2){$5$};
\draw (2,-2) circle (0.25cm) node (2,-2){$5$};
\draw (2,0) circle (0.25cm) node (2,0){$3$};
\draw (-2,0) circle (0.25cm) node (-2,0){$3$};
\draw [thick,->] ({-2+sqrt(2)/8},{2-sqrt(2)/8})--({-sqrt(2)/8}, {sqrt(2)/8});
\draw [thick,->] ({2-sqrt(2)/8},{2-sqrt(2)/8})--({sqrt(2)/8}, {sqrt(2)/8});
\end{tikzpicture}
\caption{The folded shape diagram of $A_{76}$}
\label{A76shape} 
\end{figure}
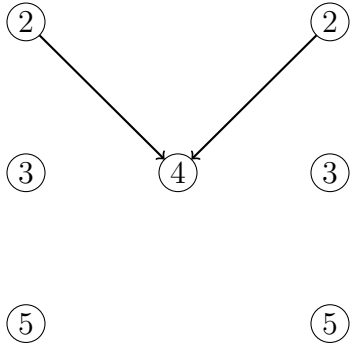
Hence the diagram gives us some true choices, which means there may be algebras with the same axet and 
Miyamoto group but of different shape---and in fact they do indeed exist. 

For our algebra $A_{76}$, knowing that it is a subalgebra of $A_{286}$, which contains no $3C$ and $4A$ in 
its shape, we can deduce that the actual shape of $A_{76}$ is $\bullet 3A3A4B$.
\end{example}

\subsection{Automorphism project} \label{package}

In this section we review our approach and the computational tools we developed for the automorphism project \cite{aut}. These tools are 
implemented in MAGMA \cite{MAGMA}. All our algebras are realised as general MAGMA algebras, although we had to 
implement our own subalgebra function because the one in MAGMA had a bug resulting in wrong results for non-associative algebras. 

The full descriptions and code for our routines are available on GitHub \cite{m11 aut}. Hence we only 
describe the main functionalities that we need to refer to below.

First of all, for smaller algebras we used the so-called \emph{nuanced algorithm}, which is fully automatic and 
can determine the full automorphism groups of algebras of Monster type $(\frac{1}{4},\frac{1}{32})$ of modest 
dimension. (The largest algebra this algorithm managed to complete was of dimension $36$.) We stress that 
this algorithm is specifically for the Monster fusion law, as it uses detailed information about the 
Norton-Sakuma algebras. For the details of the nuanced algorithm see \cite{aut}.

Our approach to finding the full automorphism group of the larger axial algebras $A$ is as follows. We first attempt to find additional axes in $A$: as twins of the known axes, as Jordan axes, and as axes corresponding to other involutions in the known group. We have efficient functions for all three of these cases, although it must be mentioned that a significant part of this efficiency comes from the assumption that the fusion law is of Monster type. Also, note that, since Jordan axes lead to additional automorphisms, we may end up with a larger automorphism group and hence the analysis of axes for known involutions should be performed after searching for Jordan axes.

We also have an efficient routine checking which automorphisms of the known automorphism group also act on $A$, which may also increase the known automorphism group.

Once we are satisfied that we likely know the full automorphism group, we switch to proving this statement. Namely, we select a small $2$-group $E$ and we aim to determine computationally the normaliser of $E$ in the full automorphism group of $A$. In practice, we aim to show that this normaliser coincides with the normaliser in the known automorphism group.

The main strategy for computing the normaliser of $E$ is decomposing the algebra $A$ with respect to the set  $Y=\{a_1,a_2,\ldots,a_m\}$ of axes corresponding to the involutions in $E$. This 
decomposition is into the direct sum of the joint eigenspaces
$$A_{(\lm_1,\lm_2,\ldots,\lm_m)}(Y)=A_{\lm_1}(a_1)\cap A_{\lm_2}(a_2)\cap\ldots\cap A_{\lm_m}(a_m)$$
of the adjoints of these axes, where $\lm_1,\lm_2,\ldots,\lm_m$ run through all combinations of eigenvalues from the fusion law; in our case they are from $\{1,0,\frac{1}{4},\frac{1}{32}\}$. 
We have a function computing this decomposition. Depending on the relation between the axes in $Y$, the sum of all joint eigenspaces may be a 
proper subspace $A_0$ of $A$, in which case we complement it by an additional term defined as the 
orthogonal complement $A_0^\perp$ in $A$ with respect to the Frobenius form. We note that the Frobenius form is positive definite in all algebras we consider here, and so $A_0^\perp$ is indeed a complement to $A_0$.

One of the summands of the decomposition, namely, the joint $0$-eigenspace $U=A_{(0,0,\ldots,0)}(Y)$, is always a subalgebra and the 
remaining pieces are modules for this subalgebra. Similarly, we have the part (or a combination of several parts) of the decomposition generated by the axes in $Y$ themselves, and that is also a subalgebra acting on the other summands. We have functions building and analysing these actions.

This decomposition is invariant under the full normaliser of $E$, up to a possible permutation of summands. Hence we also have functions building up automorphisms  gradually by extending them from $U$ to suitable additional parts of the decomposition treated as modules for $U$. Once we have extended a putative automorphism to a subspace that generates $A$, we use the special function that computes the action of this automorphism on the entire $A$. 

The difficult and time-consuming part of this approach is its starting point, namely, finding the full automorphism group of $U$. Of course, $U$ is much smaller than $A$. When possible, we find all idempotents in $U$ of a suitably chosen length by solving a system of linear and quadratic equations. 
If we have that these idempotents generate $U$, then  the automorphism group of $U$ acts faithfully on this set of idempotents and hence it can be identified in this way. This argument depends on the fact that automorphisms of $U$ preserve the available Frobenius form. We have a function that computes the space of all Frobenius forms on a given algebra. Once this space is $1$-dimensional, the form is preserved up to a scale and then having a special element, such as the identity, allows us to conclude that the Frobenius form is indeed preserved.

When $U$ is bigger, finding idempotents of a given length might be too difficult or impossible altogether. Then we simply identify a smaller invariant subalgebra $U_0$ of $U$ and we find $\Aut(U)$ by building automorphisms up starting from $U_0$ and then extending to the orthogonal complement of $U_0$ in $U$.

\medskip
This procedure has worked well in all the cases we attempted, allowing us to compute the normaliser of $E$ in $\Aut(A)$. As it happens, in all cases this normaliser coincided with the normaliser of $E$ in the known automorphism group. At this point, we switched to hand-made group-theoretic arguments analysing the possible structure of the finite group $\Aut(A)$ to deduce that it coincides with the known automorphism group. These proofs follow the same outline for all algebras $A$, focussing on the minimal normal subgroup $Q$ of $\Aut(A)$ and aiming to show that $Q$ cannot be elementary abelian using the $6$-transposition property of Miyamoto involutions in algebras of Monster type (see, for example, Corollary 2.10 in \cite{ kms}) and available representation theory of the known subgroup of $\Aut(A)$. Once $Q$ is known to be non-abelian, we reduce to the simple case and then identify $Q$ and $\Aut(A)$ using the known Sylow $2$-subgroup of $\Aut(A)$, found inside the normaliser of $E$.  

\medskip
To summarise, for larger algebras $A$, we use a hybrid approach combining a computation in MAGMA of the normaliser of $E$ with a subsequent proof of the fact that $\Aut(A)$ coincides with the known automorphism group. We have not encountered yet a situation where the normaliser of $E$ would be bigger that the known one, but if it had not been the case, we would have constructed additional automorphisms at this stage and a bigger known group, repeating our procedure after that. We also need to mention that the computational part of this approach is not fully automatic, as we need at various points to decide how to proceed: which idempotent lengths to look for, which parts of the decomposition to extend to, etc.

While this hybrid approach is similar to what we did in \cite{aut}, we had to further improve it to be able to handle the significantly larger algebra $A_{286}$. In particular, instead of $E\cong 2^2$ used in \cite{aut} and that we have here for $A_{76}$ and $A_{101}$, in the case of $A_{286}$ we focus on the non-abelian group $P\cong D_8$ instead. The general outline of the computational proof that the normaliser of $P$ in $\Aut(A_{286})$ is known, is still similar to the smaller cases, but the decomposition pieces here are much larger, and so we introduce significant new ideas for a further decomposition of the joint-zero subalgebra $U$. This further decomposition also makes our methods less reliant on the uniqueness of the Frobenius form on subalgebras and the heavy computations of all idempotents of given length. Finally, the group-theoretic part of the proof for $A_{286}$ is much more elaborate and introduces new computational methods that allow us to remove possible module structures on $Q$, in the case where $Q$ is elementary abelian, as some of those structures cannot be excluded by the more straightforward methods of \cite{aut}.

\section{Maximal subgroups $2\cdot S_4$, $S_5$, and $U_3(2){:}2$}\label{small groups}

In this section we describe algebras arising from the smaller maximal 
subgroups of $M_{11}$. All these algebras were handled
by the nuanced algorithm from \cite{aut}. 

\subsection{$2{\cdot}S_4$}

The maximal subgroup $H\cong 2{\cdot}S_4$ contains two classes of 
involutions, of lengths $1$ and $12$. The subalgebra generated by the 
$12$ axes corresponding to the involutions in the second orbit has 
dimension $17$ and hence it is denoted below as $A_{17}$.\footnote{So far as we know, this is a first appearance of
this algebra in a printed article; however, it was found earlier by J.~McInroy \cite{database} by an application
of the expansion algorithm.} It can be easily seen that the shape diagram of the group $H$ on this axet is 
connected and it contains a $6A$ node. Hence the shape $\bullet$ is unique, as 
shown in Fig \ref{A17shape}. 

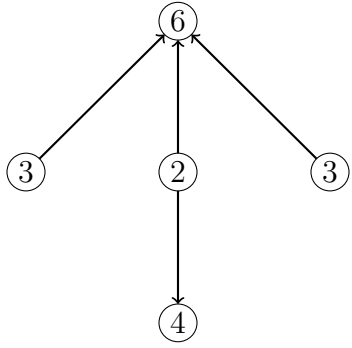
\begin{figure}[h]
\centering
\begin{tikzpicture}
\draw (-2,2) circle (0.25cm) node (-2,2){$3$};
\draw (2,2) circle (0.25cm) node (2,2){$3$};
\draw (0,4) circle (0.25cm) node (0,4){$6$};
\draw (0,2) circle (0.25cm) node (0,2){$2$};
\draw (0,0) circle (0.25cm) node (0,0){$4$};
\draw[thick, ->]({-2+sqrt(2)/8}, {2+sqrt(2)/8})--({-sqrt(2)/8},{4-sqrt(2)/8});
\draw[thick, ->]({2-sqrt(2)/8},{2+sqrt(2)/8})--({sqrt(2)/8}, {4-sqrt(2)/8});
\draw [thick,->](0,9/4)--(0,15/4);
\draw [thick,->](0,7/4)--(0,1/4);
\end{tikzpicture}
\caption{The shape diagram of $A_{17}$ on $12$ axes}
\label{A17shape}
\end{figure}

The algebra $A_{17}$ happens to also contain the axis $j$ 
corresponding to the central involution $z$ of $H$. Note that $z$ 
fixes all twelve generating axes, and hence it acts trivially on 
$A_{17}$. Thus, the Miyamoto group is just $H/\la z\ra\cong S_4$. This 
also means that $j$ is a Jordan axis in $A_{17}$, and as such, it gives an extra 
automorphism of order $2$ centralising $\Miy(A_{17})\cong S_4$. Using 
the nuanced algorithm, we checked that the resulting group $2\times 
S_4$ is the full automorphism group of $A_{17}$. We also checked 
computationally that $A_{17}$ is $2$-closed and it admits a unique, 
up to scale, Frobenius form. Clearly, it is the form inherited from 
$A_{286}$ and, in particular, it is positive definite.

The interesting feature of $A_{17}$ is that it contains twins, 
but they are within the same orbit of $12$ axes. This is a first 
example of this kind. The involution induced by $j$ switches 
twins in pairs. 

\subsection{$S_5$}

The maximal subgroup $H\cong S_5$ contains two orbits of involutions 
of lengths $10$ and $15$. The former orbit generates a subalgebra 
$A_{36}$ of dimension $36$, and this subalgebra also contains the 
orbit of length $15$.

The algebra $A_{36}$ is not new, it was mentioned in \cite{Se,axialconstruction}. Its full automorphism group was found in 
$\cite{aut}$, where it was shown, using the nuanced algorithm, that 
$\Aut(M_{36})=\Miy(A_{36})\cong S_5$. 

Note that $S_5$ has a disconnected shape diagram (see Figure 
\ref{S5_Shape}) and two possible shapes. The shape of 
$A_{36}$ on the axet $10+15$ is $\bullet 4B$, while it is $\bullet 3A2A$ on the axet 
$10$.\footnote{According to \cite{axialconstruction}, this case 
remains open, i.e., it is not known whether $A_{36}$ is the only 
algebra for $S_5$ with the axet of size $10$ and this shape.} The 
algebra $A_{36}$ is $2$-closed with respect to $10+15$ axes and it 
is $3$-closed with respect to $10$ axes. It admits a unique (positive definite) Frobenius 
form up to scale.

The orbit of length $15$ generates a $26$-dimensional subalgebra 
$A_{26}$, of shape $\bullet 3A2A$. The Miyamoto group $\Miy(A_{26})$ is 
isomorphic to $A_5$. The algebra $A_{26}$ is also not new. It has 
appeared in \cite{Se,axialconstruction}, while $\Aut(A_{26})$ was 
determined in \cite{aut}, namely, $\Aut(A_{26})\cong S_5$.

\begin{figure}[h]
\centering
\begin{tikzpicture}
\draw (0,-4) circle (0.25cm) node (0,-4){$5$};
\draw (-2,0) circle (0.25cm) node (-2,0){$3$};
\draw (2,0) circle (0.25cm) node (2,0){$3$};
\draw (0,0) circle (0.25cm) node (0,0){$6$};
\draw (2,-2) circle (0.25cm) node (2,-2){$2$};
\draw (-2,-2) circle (0.25cm) node (-2,-2){$2$};
\draw (0,-2) circle (0.25cm) node (0,-2){$2$};
\draw (-2,-4) circle (0.25cm) node (-2,-4){$4$};
\draw[thick,->](1.75,0)--(0.25,0);
\draw[thick,->](-1.75,0)--(-0.25,0);
\draw[thick,->](0,-1.75)--(0,-0.25);
\draw[thick,->](-2,-2.25)--(-2,-3.75);
\draw[thick,->]({2-sqrt(5)/10},{-2-sqrt(5)/20})--({-2+sqrt(5)/10},{-4+sqrt(5)/20});
\end{tikzpicture}
\caption{Shape diagram of $A_{36}$ on $10+15$ axes}
\label{S5_Shape}
\end{figure}
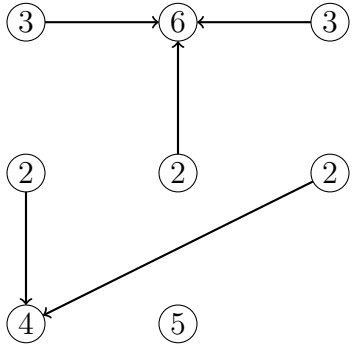

For completeness, let us mention the remaining shape $\bullet 4A$ of $S_5$ on the axet $10+15$. This completes to the unique algebra of dimension $61$, whose 
full automorphism group $2\times S_5$ (the factor $2$ induced by a Jordan 
axis) was determined in \cite{aut}.

\subsection{$U_3(2){:}2$}

Finally, let $H\cong U_3(2){:}2$. Then $H$ contains $21=9+12$ involutions. All these involutions generate a 
proper subgroup $H_0\cong S_3\wr 2$ of index $2$ in $H$. Under 
the action of $H_0$, we have three orbits of lengths, $6$, $6$, and $9$.

We will now report on the subalgebras that various combinations of 
the corresponding orbits on axes generate. The two orbits of length 
$6$ generate isomorphic subalgebras $A_{18}$ and $A'_{18}$ of 
dimension $18$. Both $A_{18}$ and $A'_{18}$ contain the orbit of 
length $9$. Furthermore, the subalgebra $A_{12}=A_{18}\cap A'_{18}$ 
is generated by this orbit. Lastly, the subspace 
$A_{24}=A_{18}+A'_{18}$ of dimension $18+18-12=24$ is in fact a 
subalgebra, and it is clearly the subalgebra generated by the two 
orbits of length $6$, as well as the entire set of $21=6+6+9$ 
axes. (See Figure \ref{A_24 Hasse}.) 

\begin{figure}[h]
\centering
\begin{tikzpicture}
\draw[thick] ({sqrt(2)/8},{2-sqrt(2)/8})--({2-sqrt(2)/8},{sqrt(2)/8});
\draw[thick] ({-sqrt(2)/8},{2-sqrt(2)/8})--({-2+sqrt(2)/8},{sqrt(2)/8});
\draw[thick] ({2-sqrt(2)/8},{-sqrt(2)/8})--({sqrt(2)/8},{-2+sqrt(2)/8});
\draw[thick] ({-2+sqrt(2)/8},{-sqrt(2)/8})--({-sqrt(2)/8},{-2+sqrt(2)/8});
\node at (0.05,2.1){$A_{24}$};
\node at (0.05,-2.1){$A_{12}$};
\node at (2.2,0){$A'_{18}$};
\node at (-2.2,0){$A_{18}$};
\end{tikzpicture}
\caption{The Hasse diagram of algebra inclusions in $A_{24}$}
\label{A_24 Hasse}
\end{figure}
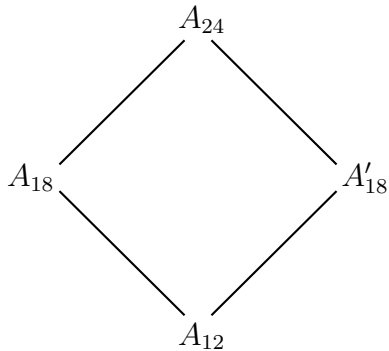

The algebras $A_{24}$ and $A_{12}$ are new. The algebra $A_{18}$ has not appeared in print before,
but it is contained in J. McInroy's database \cite{database} of algebras constructed using the expansion algorithm.

We now discuss these three algebras in turn. 

\medskip
We have mentioned above that $\Miy(A_{24})=H_0\cong S_3\wr 2$.
An application of the nuanced algorithm shows that $\Aut(A_{24})=H$, 
and in particular, $A_{24}$ contains no new axes. The shape diagram 
of $A_{24}$ with respect to the axet $6+6+9$ is shown in Figure 
\ref{shape_graphPSU} and it has two connected components, one of which has nodes of size $6$.

\begin{figure}[h]
	\centering
	\begin{tikzpicture}
		\draw (-2,2) circle (0.2cm) node (-2,2){$3$};
		\draw (0,2) circle (0.25cm) node (0,2){$3$};
		\draw (2,2) circle (0.25cm) node (2,2){$3$};
		%\draw (4,2) circle (0.25cm) node (4,2){$3$};
		\draw (-2,-2) circle (0.2cm) node (-2,-2){$6$};
		\draw (0,-2) circle (0.2cm) node (0,-2){$2$};
		\draw (2,-2) circle (0.2cm) node (2,-2){$2$};
		\draw (4,-2) circle (0.2cm) node (4,-2){$2$};
		\draw (0,0) circle (0.25cm) node (0,0){$6$};
		\draw (-2,0) circle (0.25cm) node (-2,0){$2$};
		\draw (4,0) circle (0.25cm) node (4,0){$4$};
		\draw (2,0) circle (0.25cm) node (2,0){$3$};
		\draw[thick,->] ({-2+sqrt(2)/8},{2-sqrt(2)/8})--({-sqrt(2)/8},{sqrt(2)/8});
		\draw[thick,->] ({2-sqrt(2)/8},{2-sqrt(2)/8})--({sqrt(2)/8},{sqrt(2)/8});
		\draw[->,thick] (1.75,0)-- (0.25,0);
                 \draw[->,thick](-2,-0.25)--(-2,-1.75);
                 \draw[->, thick](-1.75,0)--(-.25,0);
                 \draw[->, thick](-.25,-2)--(-1.75,-2);
		\draw[->,thick]({2+sqrt(2)/8},{-2+sqrt(2)/8})--({4-sqrt(2)/8},{-sqrt(2)/8});
		\draw[->,thick] (0,1.75)--(0,0.25);
		\draw[->,thick](0,-1.75)--(0,-0.25);
		\draw[->,thick](4,-1.75)--(4,-0.25);
	\end{tikzpicture}
	\caption{ The shape diagram of $A_{24}$ on $21$ axes}
	\label{shape_graphPSU}
\end{figure}
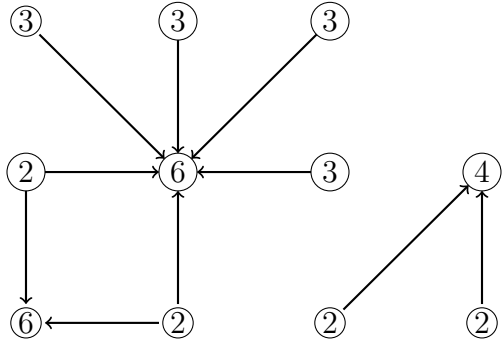
This means that there is a unique choice in this component.
The second component gives two possible choices for the shape, $\bullet 4A$ and 
$\bullet 4B$, and the algebra $A_{24}$ is of shape $\bullet 4B$. We mention for 
completeness that the algebra with the same axet and shape $\bullet 4A$ is known to 
be of dimension $45$ (see the database \cite{database}). Its full automorphism 
group is currently unknown, although it is bigger than $H_0$, as the algebra 
contains a unique Jordan axis. Also, due to the involution induced by this 
Jordan axis, both orbits in $6+6$ have twin orbits $6'+6'$, hence the total 
set of axes in this algebra contains at least $1+6+6+6'+6'+9=34$ axes. This algebra is very interesting, but we excluded it from the current project because it is not directly related to 
$A_{286}$ and $M_{11}$. Note though that we tried the nuanced algorithm on this algebra, but it was unsuccessful.

\medskip
If we instead consider the smaller axet $6+6$ then the shape diagram, shown 
in Figure \ref{A24 shape6_6}, has three components, each contributing two 
choices. 

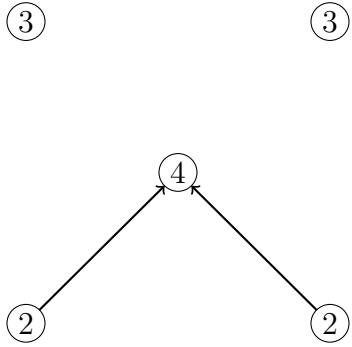
\begin{figure}[h]
	\centering
	\begin{tikzpicture}
		\draw (-2,2) circle (0.25cm) node (-2,2){$3$};
		\draw (2,2) circle (0.25cm) node (2,2){$3$};
		\draw (-2,-2) circle (0.25cm) node (-2,-2){$2$};
		\draw (2,-2) circle (0.25cm) node (2,-2){$2$};
		\draw (0,0) circle (0.25cm) node (0,0){$4$};
		\draw[thick,->]({2-sqrt(2)/8},{-2+sqrt(2)/8})--({sqrt(2)/8},{-sqrt(2)/8});
		\draw[thick,->]({-2+sqrt(2)/8},{-2+sqrt(2)/8})--({-sqrt(2)/8},{-sqrt(2)/8});
	\end{tikzpicture}
	\caption{The shape diagram of $A_{24}$ on $6+6$ axes}
	\label{A24 shape6_6}
\end{figure}
So, in total there are eight possible shapes, but only six if we take them up to isomorphism.
The shape of $A_{24}$ on $6+6$ is 
$\bullet 4B3A3A$. The algebra $A_{24}$ is $2$-closed for the axet $6+6+9$
and $3$-closed for $6+6$. Finally, $A_{24}$ admits a unique (positive definite)
Frobenius form up to scale.

It is not known whether $A_{24}$ is the only algebra with its axet and shape $\bullet 4B3A3A$. Also, the shape $\bullet 4A3A3A$ cannot currently be completed 
by the expansion algorithm, and so we do not know if any algebras of this shape exit. The four remaining shapes, $\bullet 4A3A3C$, $\bullet 4A3C3C$, $\bullet 4B3A3C$, and 
$\bullet 4B3C3C$, all collapse, so there exist no algebras of these shapes.

\medskip
The algebra $A_{18}$ is only invariant under $H_0$; however, its 
Miyamoto group is even smaller, namely, $K=\Miy(A_{18})\cong 
S_3{\times} S_3$. Under this group's action, the orbit $6$ splits as
$3+3$, while $9$ remains a single orbit. The shape diagram on $3+3+9$ 
is shown in Figure \ref{A18shape_15}. It allows four shapes. 
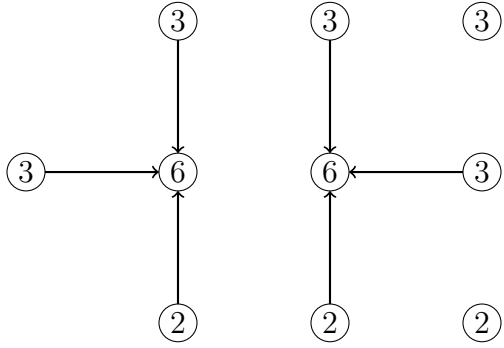
\begin{figure}[h]
	\centering
	\begin{tikzpicture}
		\draw (0,0) circle (0.25cm) node (0,0){$6$};
		\draw (2,0) circle (0.25cm) node (2,0){$6$};
		\draw (4,0) circle (0.25cm) node (4,0){$3$};
		\draw (0,2) circle (0.25cm) node (0,2){$3$};
		\draw (2,2) circle (0.25cm) node (2,2){$3$};
		\draw (4,2) circle (0.25cm) node (4,2){$3$};
		\draw (0,-2) circle (0.25cm) node (0,-2){$2$};
		\draw (2,-2) circle (0.25cm) node (2,-2){$2$};
		\draw (4,-2) circle (0.25cm) node (4,-2){$2$};
                \draw[->,thick](3.75,0)--(2.25,0);
                \draw [->, thick] (0, -1.75)--(0, -.25);
		\draw (-2,0) circle (0.25cm) node (-2, 0){$3$};
		\draw[->, thick] (-1.75,0)--(-.25, 0); 
		\draw [->, thick] (0, 1.75)--(0, .25);
		\draw [->, thick] (2, 1.75)--(2, .25);
                \draw [->, thick] (2, -1.75)--(2, -.25);
	\end{tikzpicture}
	\caption{The shape diagram of $A_{18}$ on $3+3+9$ axes}
	\label{A18shape_15} 
\end{figure}
The shape 
of $A_{18}$ is $\bullet 3A2A$ and it is the only algebra with this shape. Of 
the remaining three shapes, $\bullet 3C2A$ and $\bullet 3C2B$ collapse, while $\bullet 3A2B$ leads to a unique algebra of dimension $25$.
All of these appeared in \cite{axialconstruction}. The full automorphism groups of both $A_{18}$ 
and the $25$-dimensional algebra have been determined in \cite{aut} using 
the nuanced algorithm. In particular, $\Aut(A_{18})=H_0\cong S_3\wr 2$. 
Again, the Frobenius form on $A_{18}$ is positive definite and unique up to scale.

If we again remove the orbit $9$ and consider the axet $3+3$ for $A_{18}$ 
then the shape diagram is as in Figure \ref{A18shape_6}, which indicates eight possible shapes. However, up to isomorphism there are only six possible shapes on $3+3$. 
\begin{figure}[h]
	\centering
	\begin{tikzpicture}
		\draw (-2,0) circle (0.25cm) node (-2,0){$3$};
		\draw (0,0) circle (0.25cm) node (0,0){$3$};
		\draw (2,0) circle (0.25cm) node (2,0){$2$};
	\end{tikzpicture}
	\caption{The shape diagram of $A_{18}$ on $3+3$ axes}
	\label{A18shape_6}
\end{figure}
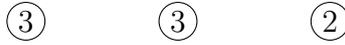
The algebra $A_{18}$ has shape $\bullet 3A3A2A$ and it is not known whether it is the only one with this shape. The shapes $\bullet 3A3C2A$ and $\bullet 3C3C2A$ collapse. 
The three shapes involving $2B$, namely, $\bullet 3A3A2B$, $\bullet 3A3C2B$, and $\bullet 3C3C2B$ 
all lead to direct sum algebras, $3A\oplus 3A$ (dimension $8$), $3A\oplus 3C$ (dimension $7$), and $3C\oplus 3C$ (dimension $6$), which are all well understood.

\medskip
Finally, the algebra $A_{12}$, generated by the orbit $9$, has Miyamoto group $L=3^2{:}2$.
Figure \ref{A12shape_9} represents the shape diagram of the corresponding axet. 
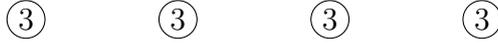
\begin{figure}[h]
	\centering
	\begin{tikzpicture}
		\draw (-2,0) circle (0.25cm) node (-2,0){$3$};
		\draw (0,0) circle (0.25cm) node (0,0){$3$};
		\draw (2,0) circle (0.25cm) node (2,0){$3$};
		\draw (4,0) circle (0.25cm) node (4,0){$3$};
	\end{tikzpicture}
	\caption{The shape diagram of $A_{12}$ on $9$ axes}
	\label{A12shape_9}
\end{figure}
Here we have four binary choices and hence $16$ shapes, but only five up to isomorphism: $\bullet 3A3A3A3A$, $\bullet 3A3A3A3C$, $\bullet 3A3A3C3C$, $\bullet 3A3C3C3C$, and $\bullet 3C3C3C3C$.
(Note that the automorphism group of the axet is $AGL(2,3)$.) The algebra $A_{12}$ is of shape $\bullet 3A3A3A3A$ and its full 
automorphism group, found by the nuanced algorithm, is $AGL(2,3)$. Here 
$L=\Miy(A_{12})$ has index $24$ in the full automorphism group, the biggest 
such index known to date. It is currently not known whether $A_{12}$ is the 
only algebra with its axet and shape.

The shape $\bullet 3A3A3A3C$ collapses. We have no further information about the 
shape $\bullet 3A3A3C3C$. The shape $\bullet 3A3C3C3C$ completes to a $12$-dimensional 
algebra, different from $A_{12}$ and appearing in the database 
\cite{database}, and $\bullet 3C3C3C3C$ to a $9$-dimensional algebra, this latter 
one being, clearly, the algebra of Jordan type $\frac{1}{32}$, i.e., the 
Matsuo algebra of the $3$-transposition group $L\cong 3^2{:}2$. 

For the $12$-dimensional algebra of shape $\bullet 3A3C3C3C$, the nuanced algorithm 
shows no further axes, which means that its automorphism group acts 
faithfully on the axet $9$ and the Miyamoto group is $L\cong 3^2{:}2$. This 
yields that the full automorphism group of this algebra coincides with the 
stabiliser in $\Aut(3^2{:}2)\cong AGL(3,2)$ of the shape, and so it is an extension 
of $L$ by $S_3$ with an overall structure $3{.}S_3^2$. The $9$-dimensional Matsuo
algebra for $L$ was similarly checked with the nuanced algorithm, which showed that the algebra contains no further axes. 
Therefore, the full automorphism of this algebra is simply $AGL(3, 2)$, as the 
shape $\bullet 3C3C3C3C$ is stabilised by this entire group. 

\medskip
The information about the algebras $A_{24}$, $A_{18}$ and $A_{12}$ is 
summarised in Table \ref{results}

\section{The $76$-dimensional algebra for $A_6$ of shape $\bullet 4B3A3A$} \label{M10 alg}

In this section we determine the full automorphism group of the $76$-dimensional algebra $A=A_{76}$ 
corresponding to the maximal subgroup $G_0\cong M_{10}$ of $M_{11}$. Note that $G_0$ induces a group 
of automorphisms of $A$, though the Miyamoto group of $A$ is just $M\cong A_6$ of index $2$ in $G_0$. 
The known axet $X$ in $A$ consists of $45$ axes. The algebra $A$ is $2$-closed and admits a unique (up to scale) positive 
definite Frobenius form, which is the restriction to $A$ of the Frobenius form of the larger
algebra $A_{286}$. Let $T\cong D_8$ be a Sylow $2$-subgroup of $M$, and let 
$E\cong 2^2$ be a subgroup of $T$. Set $a_1$, $a_2$ and $a_3$ to be the axes whose Miyamoto involutions 
lie in $E$. Let $Y=\{a_1,a_2,a_3\}.$ The axes in $Y$ span a subalgebra $V=\la 
Y\ra\cong 2A$. We record some basic facts about the algebra $A$ and its joint eigenspace decomposition with respect to $Y$.

\begin{computation}
\begin{enumerate}
\item The algebra $A$ does not have Jordan axes;
\item the $45$ axes in $X$ do not have twins; and
\item $G:=\Aut(M)$ acts on $A$.
\end{enumerate}
\end{computation}

Hence the group of known automorphisms of $A$ is $G=\Aut(A_6)\cong A_6.2^2$, which is twice bigger than $G_0$.
The group $G$ has three classes of involutions of sizes $45$, $30$ and $36$. We check whether 
the involutions in $G\setminus M$ (the last two classes) are not induced by axes.

\begin{computation}
Involutions in $G\setminus M$ do not correspond to axes of Monster type $\mathcal{M}\left(\frac{1}{4}, \frac{1}{32}\right)$ in $A$.
\end{computation}

We will now decompose $A$ with respect to $Y$.

\begin{computation} 
The joint eigenspace decomposition of $A$ corresponding to $Y$ has components as follows: 
\begin{enumerate}
\item $U:= A_{(0,0,0)}(Y)$ is of dimension $15$;
\item the remaining nonzero summands $A_{(\lm_1, \lm_2,\lm_3)}(Y)$ are:
\begin{enumerate}
\item $A_{\left(\frac{1}{4},\frac{1}{32},\frac{1}{32} \right)}(Y)$, $A_{\left(\frac{1}{32},\frac{1}{4},\frac{1}{32} \right)}(Y)$, 
and $A_{\left(\frac{1}{32},\frac{1}{32},\frac{1}{4} \right)}(Y)$, of dimension $5$ each and;
\item $A_{\left(0,\frac{1}{32},\frac{1}{32} \right)}(Y)$, $A_{\left(\frac{1}{32},0,\frac{1}{32} \right)}(Y)$, 
and $A_{\left(\frac{1}{32},\frac{1}{32},0\right)}(Y)$, each of dimension $11$. 
\end{enumerate}
\end{enumerate}
\end{computation}

Note that the sum of all these joint eigenspaces is not all of $A$, which is normal when we decompose a big algebra with
respect to a subalgebra $2A$ (c.f. \cite{aut}).

\medskip
We begin with determining $\Aut(U)$. Since the dimension of $U$ is not small enough, finding all idempotents 
in $U$ is hopeless with the computational tools we currently have. Hence we will explore ways of getting 
idempotents in $U$ of specific lengths, which then 
can be found fast. Let $N\cong 2\times S_4$ be the normaliser in $G$ of $E$ and consider the fixed subalgebra $U_N$ of 
$N$ in $U$. This subalgebra is of dimension $5$ and we can easily find all idempotents in it. It turns out that $U_N$ has
$24$ idempotents of distinct square lengths. In particular, there is a unique idempotent $d$ of length 
$2$ in $U_N$. 

\begin{computation}
\begin{enumerate}
\item Up to scale, $U$ admits a unique Frobenius form;
\item the idempotent $d$ is the only one in $U$ having length $2$;
\item $\rm{ad}_d$ has eigenvalues $1$, $0$, $\frac{1}{2}$ and $\frac{1}{8}$ of dimensions $1$, $34$, $9$ and $32$, respectively, in $A$; and $1$, $7$, $3$ and $4$ in $U$;
\item $d$ satisfies an almost Monster fusion law $\mathcal{AM}\left(\frac{1}{2},\frac{1}{8}\right)$ in $U$ and $A$. 
\end{enumerate}
\end{computation}

\begin{remark}
We observe that the almost Monster fusion law is $C_2$-graded and so we have a Miyamoto involution $z:=\tau_d$ corresponding to $d$.
This involution $z$ is contained in $G\setminus M$ in the class of size $30$. In addition, $A= \lla d^G\rra$ so that we can regard 
$A$ as an axial algebra with Miyamoto group $S_6$ on an axet of size $30$ with almost Monster $\mathcal{AM}\left(\frac{1}{2}, \frac{1}{8}\right)$ 
fusion law.
\end{remark}

Note also that, since $N$ fixes $d$, it centralises $z$, which means that $z$ is the central involution from $N$.

We will use $d$ to decompose $U$ to help us establish the full automorphism group of $U$. We start with $U'=U_0(d)$. 
First, we establish the following computational facts.

\begin{computation}\label{comp U0d}
\begin{enumerate}
\item Up to scale, $U'$ admits a unique Frobenius form;		
\item the subalgebra $U'$ has precisely three idempotents $v_1$, $v_2$ and $v_3$ 
of length $\frac{17}{5}$;
\item $\lla v_1,v_2,v_3\rra =U'$;
\item the adjoints of the $v_i$ have eigenvalues $1$, $0$, $\frac{1}{10}$, $\frac{13}{30}$, 
$\frac{1}{20}$ and $\frac{7}{20}$ with respective multiplicities $1$, $2$, $1$, $1$, $1$, and $1$; and 
\item idempotents $v_i$ satisfy a $C_2$-graded fusion law on $U$ with 
\[U'_+(v_i)=\bigoplus_{\lm\in \bigl\{1,0,\frac{1}{10},\frac{13}{30}\bigr\}}U'_\lm(v_i)\]
and
\[U'_-(v_i)=\bigoplus_{\lm\in \bigl\{\frac{1}{20},\frac{7}{20}\bigr\}} U'_\lm(v_i).\]
In addition, $\la \tau_{v_i}\mid i=1,2,3\ra\cong S_3$ on $U'$.
\end{enumerate}
\end{computation}

This gives us the following lemma. 

\begin{lemma}\label{autU'lemma}
The full automorphism group $\Aut(U')$ of $U'$ is isomorphic to $S_3$.
\end{lemma}

\begin{proof}
By Computation \ref{comp U0d}(c), any automorphism $\phi\in \Aut(U')$ is completely determined by its 
effect on $v_1$, $v_2$ and $v_3$, since they generate $U'$. Part (e) shows that the tau maps $\tau_{v_i}$
induce $S_3$ on $\{v_1, v_2, v_3\}$, and hence on $U'$.

Every automorphism of $U'$ preserves the set $\{v_1,v_2,v_3\}$ since it preserves the length by Computation 
\ref{comp U0d}(a)and the $v_i$ are the only idempotents of length $\frac{17}{5}$ 
in $U'$ by part (b). Since the $\tau_{v_i}$ induce the full symmetric
group on $\{v_1, v_2,v_3\}$, $\Aut(U')=\la \tau_{v_i}\mid i=1,2,3\ra$. 
\end{proof}

Note that $\Aut(U)$ fixes $d$ and hence it acts on $U'$.

\begin{computation}\label{ext 7 to 15}
Let $W=U_{\frac{1}{8}}(d)$. Then the following holds.
\begin{enumerate}
\item $\lla W\rra =U$;
\item The identity automorphism on $U'$ extends as a $1$-dimensional space of extensions on $W$; and
\item for some nonzero element $w\in W$ and a nonzero $u'\in U'$, $(w^2, u')\ne 0$. 
\end{enumerate}
\end{computation}

\begin{lemma}\label{kernel U76}
The kernel of $\Aut(U)$ acting on $U'$ is $\la \tau_d\ra \cong C_2$.
\end{lemma}

\begin{proof}
Suppose that $\phi\in \Aut(U)$ acts as identity on $U'$. By Computation \ref{ext 7 to 15} (b), 
$\phi_{|_W}=\lm\rm{id}_W$, for some scalar $\lm$. Part (c) then gives $0\ne (w^2, u')$ so 
that $0\ne (w^2, u')=((w^2)^\phi, u'^\phi)=((w^\phi)^2, u')=((\lm w)^2,u')=\lm^2(w^2,u')$. 
We conclude that $\lm ^2=1$ and $\lm =\pm 1$. This shows that the kernel of $\Aut(U)$ acting on $U'$ 
has order at most $2$. On the other hand, $\tau_d$ is in this kernel, and so the claim holds.
\end{proof}

We need the following further computational facts to determine $\Aut(U)$. 

\begin{computation}\label{seven fifths}
\begin{enumerate}
\item $U$ has exactly three idempotents $u_1$, $u_2$ and $u_3$ of length $\frac{7}{5}$;
\item the adjoints of the $u_i$ have eigenvalues $1$, $0$, $\frac{3}{10}$ and $\frac{1}{20}$ with multiplicities $1$, $6$, 
$3$ and $5$, respectively;
\item the $u_i$ satisfy the almost Monster fusion law $\mathcal{AM}\left(\frac{3}{10},\frac{1}{20}\right)$ in $U$; and
\item  $\la \tau_{u_i}\mid i=1,2,3\ra \cong S_3$.
\end{enumerate}
\end{computation}

This gives us the following lemma. 

\begin{lemma}
$\Aut(U)\cong 2\times S_3$.
\end{lemma}

\begin{proof}
By Computation \ref{seven fifths}(d),  $\la \tau_{u_i}\mid i=1,2,3\ra$. Since $d$ is unique in $U$ of its length,
$\tau_d$ is central in $\Aut(U)$, giving us a subgroup $2\times S_3$ of $\Aut(U)$. On the other hand, by Lemmas 
\ref{autU'lemma} and \ref{kernel U76}, the order of $\Aut(U)$ does not exceed $12$, so the claim follows.
\end{proof}

\begin{remark}
Alternatively, instead of this proof, we could have simply extended the known automorphisms of $U'$ to $W$
to show that each of them has exactly two extensions.
\end{remark}

We now consider an automorphism $\phi$ acting trivially on $U$ and we want to establish how it can extend to the entirety of $A$. The focus will be on the $11$-dimensional components $W_1= A_{\left(0, \frac{1}{32}, \frac{1}{32}\right)}(Y)$, $W_2= A_{\left(\frac{1}{32}, 0,\frac{1}{32}\right)}(Y)$, and $W_3= A_{\left(\frac{1}{32},\frac{1}{32},0\right)}(Y)$. We first show that each extension $\phi$ leaves them invariant.

Recall that each $W_i$ is a module for $U$.

\begin{computation}\label{matching vi}
Up to reordering $v_1$, $v_2$ and $v_3$, we have that $v_i$ acting on $W_i$ has eigenvalues $\frac{1}{5}$, $\frac{3}{10}$, $\frac{1}{20}$, $\frac{9}{20}$, $\frac{19}{20}$, $\frac{1}{30}$, and $\frac{23}{60}$, 
with multiplicities $1$, $2$, $4$, $1$, $1$, $1$, and $1$; on the other two $W_j$, 
$v_i$ has eigenvalues  $\frac{1}{32}$, $\frac{37}{160}$, $\frac{53}{160}$, $\frac{117}{160}$, $\frac{31}{480}$, and $\frac{79}{480}$, 
with multiplicities $2$, $2$, $4$, $1$, $1$, and $1$,  respectively.
\end{computation}

This means that we have a natural matching between the idempotents $v_1$, $v_2$, and $v_3$ and the components $W_1$, $W_2$ and $W_3$. In particular, since $\phi$ fixes $v_1$, $v_2$, and $v_3$, we have the following.

\begin{lemma}
Every $\phi\in\Aut(A)$ that normalises $E$ and acts as identity on $U$ preserves $W_1$, $W_2$, and $W_3$.
\end{lemma}

We now investigate how $\phi$ extends to these components.

\begin{computation}\label{extensions A6}
\begin{enumerate}
\item The identity automorphism on $U$ admits $1$-dimensional spaces of extensions on each $W_i$;
\item for any pair $\{i, j\}\subset \{1,2,3\}$, $\lla W_i, W_j\rra =A$;
\item for some nonzero elements $w_i\in W_i$ and $u\in U$, we have the following:
\begin{enumerate}
\item $(w_i^2, u)\ne 0$ and;
\item $(w_1w_2, w_3) \ne 0$.
\end{enumerate}
\end{enumerate}
\end{computation}

We now explore the consequences of the above computations. By Computation \ref{extensions A6} (a), $\phi$ acts as a scalar $\nu_i$ on 
$W_i$. Let $\hat{N} =N_{\Aut(A)}(E)$ be the normaliser in $\Aut(A)$ of $E$. The computation above leads us to the following.

\begin{lemma}\label{kernel N76}
The kernel $\hat K$ of $\hat N$ acting on $U$ is $E$. 
\end{lemma}

\begin{proof}
Let $\phi\in \hat K$. By Computation \ref{extensions A6} (c)(i), $0\ne (w_i^2, u)=((w_i^2)^\phi, u^\phi)=((w_i^\phi)^2, u)=((\nu_i w_i)^2, u)=\nu_i^2(w_i^2, u)$. It follows that $\nu_i^2=1$ and so $\nu_i=\pm 1$. 

In a similar manner, Computation \ref{extensions A6} (c)(ii) gives $0\ne (w_1w_2, w_3)=((w_1w_2)^\phi, w_3^\phi)=(w_1^\phi w_2^\phi, \nu_3 w_3)=((\nu_1 w_1)(\nu_2 w_2), \nu_3 w_3)=\nu_1\nu_2\nu_3(w_1w_2, w_3)$ from which we conclude that $\nu_1\nu_2\nu_3 =1$.
This means that in the triples $(\nu_1,\nu_2,\nu_3)$ negative values appear in even numbers. Thus, $(\nu_1, \nu_2,\nu_3)$ lies in the set 
$$\{(1,1,1), (1, -1, -1), (-1, 1,-1), (-1, -1,1) \}\cong 2^2.$$
The first element corresponds to the identity, the second to $\tau_{a_1}$, the third to $\tau_{a_2}$ and the fourth to $\tau_{a_3}$. Thus $\hat K =E$ as claimed. 
\end{proof}

From the lemma above, the following holds.

\begin{corollary}
$\hat N =N$.
\end{corollary}

\begin{proof}
Since $\hat K =E$ by Lemma \ref{kernel N76}, $N\leq \hat N$ induces $N/E\cong 2\times S_3$ on $U$ which is the full automorphism group of $U$ so the result follows.
\end{proof}

We will use group-theoretic arguments to establish the ultimate theorem of this section.

\begin{theorem}
We have $\Aut(A_{76}) =G\cong\Aut(A_6)$.
\end{theorem}

As it happens, the argument from  \cite[Theorem 13.17]{aut}, where we find the automorphism group of another axial algebra for $\Aut(A_6)$, works {\it mutatis mutandis}. Hence we only provide an outline here. 

Let $\hat G=\Aut(A)$. We need to show that $\hat G$ coincides with its known subgroup $G\cong\Aut(A_6)$. By the above calculation, we know that $\hat N=N_{\hat G}(E)=N_G(E)=N\cong 2\times S_4$. The first step is to show that $G$ and $\hat G$ share a Sylow $2$-subgroup. This is a standard argument in finite group theory. Let $\hat S$ be a Sylow $2$-subgroup of $\hat G$ containing a Sylow $2$-subgroup $S$ of $G$, with $E\leq S$. If $S<\hat S$ then also $S<N_{\hat S}(S)$ and this leads eventually to additional elements normalising $E$, a contradiction, since $\hat N=N$.

After this we consider a minimal normal subgroup $Q$ of $\hat G$. We first eliminate the possibility that $Q$ is an elementary abelian $p$-group. The case $p=2$ is impossible because $S$ would then have a nontrivial centraliser in $Q$, but $S$ is self-centralised and it intersects $Q$ trivially in this case. The odd $p$ are eliminated via the $6$-transposition property of Miyamoto involutions. Clearly, this immediately forces $p=3$ or $5$. For these two primes, we consider an irreducible $M$-submodule of of $Q$ and dispose of all possibilities using the available modular character tables of $M\cong A_6$. We eliminate all nontrivial modules (the trivial module is not possible due to $S$ being self-centralised) using the following handle: two Miyamoto involutions from $M$, whose product has order $\{3,5\}\setminus\{p\}$,  cannot invert the same nontrivial element of $Q$, as this would lead to a product of order $15$. This in turn implies a weaker but simpler condition that the dimension of a possible irreducible module must be a multiple of $3$. For $p=5$, this eliminates everything, and for $p=3$, the combination of the simple condition and the more elaborate one also finishes the job.

At this point, we know that $Q$ is a direct product of isomorphic non-abelian simple groups. Since $G$ contains  the Sylow $2$-subgroup $S$ of $\hat G$, we deduce that $M\leq Q$, and also, since $S$ is small, we deduce that $Q$ is itself a simple group. Finally, we apply the classification \cite{GW} of finite simple groups with a dihedral Sylow $2$-subgroup to show that $Q$ can only be $L_2(q)$, with $q\equiv 7,9\mbox{ mod }16$, or $A_7$, and after again using the $6$-transposition condition and other available information, we eventually deduce $Q=M$ and $\hat G=G$.

\section{The $101$-dimensional algebra for $L_2(11)$ of shape $\bullet$}\label{L211 algebra}

This algebra $A=A_{101}$ was constructed from the class of $55$ involutions in $G:=PGL(2,11)$ with the corresponding Miyamoto group $G_0:=L_2(11)$.
Let $S\cong D_8$, the dihedral group of order $8$, be a Sylow $2$-subgroup of $G$. Then $S$ has three classes of involutions,
one central, and two of size two each. The central involution, $\tau_1$, say, is induced by an axis $a_1$, and one of
the two classes is also induced by axes. Let the class induced by axes be $\{\tau_2, \tau_3\}$, and let
$a_2$ and $a_3$ be the corresponding axes. Set $E=\la\tau_1,\tau_2,\tau_3\rangle \cong 2^2$. This is a Sylow $2$-subgroup of $G_0$. 
The corresponding subalgebra $V=\la a_1, a_2, a_3\ra$ is isomorphic to $2A$. 
 
We first check for the additional axes and uniqueness of the Frobenius form.
 
\begin{computation}\label{prelim_comp_L2_11}
\begin{enumerate}
\item The algebra $A$ has no Jordan axes;
\item the $55$ known axes do not have twins;
\item the involutions in $G\setminus G_0$ are not induced by axes;
\item the positive definite Frobenius form on $A$ is unique up to scale. 
\end{enumerate}
\end{computation}

Let $N_0=N_{G_0}(E)\cong A_4$ and $N=N_G(E)\cong S_4$. The group
$G_0=L_2(11)$ has two conjugacy classes of subgroups of order $60$ isomorphic to $A_5$. 
Let $H_1$ and $H_2$ be the (unique) members of these classes which contain $N_0$. Let $B_i$ be the subalgebra of $A$ generated by the axes corresponding to the $15$ involutions from $H_i$. 
Then $B_i$ is of dimension $26$. Looking at the shape diagram of $A_{101}$ (see Figure \ref{L211 shape}), the shape of $B_i$ can only be $\bullet 3A2A$ and so the dimension information agrees with \cite{Se, axialconstruction}.
 
\medskip
Let $Y=\{a_1, a_2, a_3\}$. We will now decompose $A$ with respect to $Y$. 

\begin{computation}
The joint eigenspace decomposition of $A$ with respect to $Y$ is as follows:
\begin{enumerate}
\item $U=A_{\left(0,0,0\right)}(Y)$ of dimension $18$;
\item $A_{\left(\frac{1}{4},\frac{1}{32},\frac{1}{32}\right)}(Y),A_{\left(\frac{1}{32},\frac{1}{4},\frac{1}{32}\right)}(Y)$ and
$A_{\left(\frac{1}{32},\frac{1}{32},\frac{1}{4}\right)}(Y)$, each of dimension $5$; and
\item $A_{\left(0, \frac{1}{32}, \frac{1}{32}\right)}(Y), A_{\left(\frac{1}{32},0,\frac{1}{32}\right)}(Y)$
and $A_{\left(\frac{1}{32},\frac{1}{32},0 \right)}(Y)$, each of dimension $17$.
\end{enumerate}
\end{computation}

As usual, we start with determining $\Aut(U)$. First of all, note that $N=N_G(E)$ acts on $U$ and induces on it a factor group of the group $\bar N=N/E\cong S_3$, as $E$ acts trivially on $U$. It is easy to check computationally that the elements of order $3$ from $N$ act on $U$ nontrivially, which means that the whole group $\bar N\cong S_3$ is induced on $U$. We aim to show that $\Aut(U\cong\bar N$.

\medskip
Set $b_i:=\one_{B_i}-\one_{V}$, $i=1,2$. By the definitions of $B_i$, we have that $V\subseteq B_i$, and hence
$b_i\in U$ for $i=1,2$. Indeed, $b_ia_j=\left(\one_{B_i}-\one_{V}\right)a_j=a_j-a_j=0$ for $a_j \in Y$. To establish the invariance
of $\lla b_1, b_2\rra$ under the automorphisms normalising $E$, we have the following calculation in $U$.

\begin{computation} \label{only two}
The elements $b_1$ and $b_2$ are the only two idempotents of length $\frac{276}{35}$ in $U$.
\end{computation}

Let $W=\lla b_1, b_2\rra$. Then  $W$ has dimension $6$ and is invariant under automorphisms of $A$ that normalise $E$. We note that involutions from $S\setminus E$ permute $H_1$ and $H_2$ and consequently they also permute $b_1$ and $b_2$.

\begin{computation} \label{for x}
\begin{enumerate}
\item The eigenspace $A_1(b_i)=U_1(b_i)$ is $4$-dimensional and $W^\prime:=\cap _iU_1(b_i)$ is $1$-dimensional;
\item the unique nonzero  idempotent $x$ in $W^\prime$ is of length $\frac{20}{7}$;
\item $\mathrm{ad}_x$ has spectrum $1$, $0$, $\frac{1}{14}$, $\frac{2}{7}$ and $\frac{3}{7}$ with the corresponding eigenspaces having dimensions $1$, $5$, $6$, $2$ and $4$, respectively.
\item $x$ satisfies a Seress, ungraded fusion law in $U$; and
\item $U_1(x)+U_0(x)=W$.
\end{enumerate}
\end{computation}

We use the decomposition with respect to $x$ to find $\Aut(U)$. Let $T=U_{\frac{1}{14}}(x)$.

\begin{computation}\label{T_generatesU}
$U=\lla T\rra$.
\end{computation}

In light of Computations \ref{only two}, \ref{for x} and  \ref{T_generatesU}, all the automorphisms of $U$ are completely determined  by their action on $T$. Since $\Aut(U)$ leaves the generating set $\{b_1,b_2\}$ of $W$ invariant, and furthermore, every element of $S\setminus E$ switches $b_1$ and $b_2$, we conclude that $\Aut(U)$ induces on $W$ a group of order $2$. It remains to find the kernel of the action of $\Aut(U)$ on $W$. Hence we focus on extensions of the identity automorphism of $W$ to $T$.

Note that, as $x$ satisfies a Seress fusion law, $T$ is a module for $W$.

\begin{computation}
\begin{enumerate}
\item The module $T$ admits a $2$-dimensional space $H$ of extensions of the identity automorphism of $W$; 
\item a non-identity $\phi\in H$ has an irreducible quadratic minimal polynomial $f$ over $\Q$.
\end{enumerate}
\end{computation}

Let $\tilde\Q$ be the quadratic extension of $\Q$ corresponding to this minimal polynomial. Let $\tilde A$ be the algebra $A$ viewed over $\tilde\Q$ and we similarly adapt the tilde notation for all subalgebras and subspaces of $A$.

\begin{computation}\label{extensions_Ti}
\begin{enumerate}
\item Acting on $\tilde T$, the nontrivial $\phi\in H$ has two conjugate eigenvalues $\xi_1$ and $\xi_2$, 
and the corresponding eigenspaces $T_i$, $i=1,2$, of $\phi$ are each of dimension $3$; 
\item $T_1$ and $T_2$ are modules for $\tilde W$; and
\item each $T_i$ has a $1$-dimensional space of extensions of the identity automorphism on $\tilde W$.
\end{enumerate}
\end{computation}

We note that while $T_1$ and $T_2$ have been defined in terms of one specific $\phi\in H$, these subspaces of $\tilde T$ are invariant under the whole $H$. Hence we can again now talk about an arbitrary $\phi\in H$.  

A consequence of Computation \ref{extensions_Ti}(c) is that any $\phi\in H$ acts as a scalar
$\lm_i$ on $T_i$, $i=1,2$. 

\begin{computation} \label{T}
\begin{enumerate}
\item For some $t_i\in T_i$ and $w\in \tilde W$, we have $(t_i^3,w)\neq 0$; and
\item $(t_1t_2,w)\neq 0$.
\end{enumerate}
\end{computation}

We can now prove the following lemma.

\begin{lemma}\label{lm_i}
The kernel of $\Aut(U)$ acting on $W$ is of order $3$.
\end{lemma} 

\begin{proof}
Since $\bar N\leq\Aut(U)$, we know that the kernel has order at least $3$. Hence we need to show that it has order at most $3$. Let $\psi$ be an element of the kernel, and let $\phi:=\psi|_T\in H$.  In view of Computation \ref{T}(a),
$$
0\neq (t_i^3,w)=((t_i^3)^\psi, w^\psi)=((t_i^\phi)^3,w)=((\lm_i t_i)^3,w)
=\lm_i^3(t_i^3,w).
$$
It follows that $\lm_i^3=1$. Also, by part (b), 
$$
0\neq(t_1t_2,w)=((t_1t_2)^\psi,w^\psi)=(t_1^\phi t_2^\phi,w)=(\lm_1 t_1\lm_2 t_2,w)=\lm_1\lm_2(t_1 t_2,w),
$$
and hence $\lm_1\lm_2=1$, that is, $\lm_2=\lm_1^{-1}$. It follows we have only three possible $\phi:=\psi|_T$, and so no more than three elements $\psi$ in the kernel of $\Aut(U)$ acting on $W$. 
\end{proof}

Hence we now have the desired claim.

\begin{corollary} \label{full U}
The automorphism group $Aut(U)\cong\bar N\cong S_3$.
\end{corollary}

Clearly, by Lemma \ref{lm_i}, the order of $\Aut(U)$ is at most $6$, and so $\Aut(U)$ indeed coincides with the group $\bar N\cong S_3$ induced by $N=N_G(E)$.

\medskip 
We now seek to extend the automorphisms of $U$ to the entire algebra $A$. Due to the following computational facts, we concentrate on the $17$-dimensional summands
$W_1:=A_{\left(0,\frac{1}{32},\frac{1}{32}\right)}(Y)$, $W_2:=A_{\left(\frac{1}{32},0,\frac{1}{32}\right)}(Y)$, and $W_3:=A_{\left(\frac{1}{32},\frac{1}{32},0\right)}(Y)$. As usual, these summands are modules for $U$.

\begin{computation}\label{extensions_U_to_A_L2_11}
\begin{enumerate}
\item The summands $W_i$ generate $A$, i.e., $A=\lla W_1, W_2, W_3\rra$; in fact, any pair of them generates $A$;
\item each $W_i$, $i=1,2,3$, admits a $1$-dimensional space of extensions of the identity automorphism of $U$;
\item for some $w_i\in W_i$ and $u\in U$, we have 
\begin{enumerate}
\item $(w_i^2,u)\neq 0$ and;
\item $(w_1w_2, w_3)\neq 0$.
\end{enumerate}
\end{enumerate}
\end{computation}

We are now able to identify the kernel $K$ of $\hat N:=N_{\Aut(A)}(E)$ acting on $U$.

\begin{lemma} \label{U kernel}
We have that $K=E$.
\end{lemma}

\begin{proof}
By Computation \ref{extensions_U_to_A_L2_11} (a), any automorphism of $A$ stabilising $Y$ is completely determined by its action on the components $W_i$. Let $\phi\in K$. By part (b), $\phi$ acts on each $W_i$ as a scalar $\mu_i$.
By (c)(i), $\mu_i^2=1$, since $0\neq (w_i^2,u)=((w_i^2)^\phi, u^\phi)=((w_i^\phi)^2,u)=((\mu_iw_i)^2,u)=\mu_i^2(w_i^2,u)$. Hence $\mu_i=\pm 1$.
Similarly, (c)(ii) yields $\mu_1\mu_2\mu_3=1$. This gives us the following possibilities for triples $(\mu_1,\mu_2, \mu_3)$: 
$$\{(1,1,1), (1,-1,-1), (-1,1-1), (-1,-1,1)\}.$$ Clearly, the first of these is realised by the identity automorphism, while the remaining three are realised by the
Miyamoto involutions $\tau_1$, $\tau_2$ and $\tau_3$ of $a_1$, $a_2$, and $a_3$, respectively. The result follows.
\end{proof}

Consequently, we know the full normaliser of $E$ in $\Aut(A)$.

\begin{corollary}
We have that $\hat N=N_{\Aut(A)}(E)=N$.
\end{corollary} 

Indeed, Corollary \ref{full U} and Lemma \ref{U kernel} imply that $|\hat N|=24=|N|$, and so they must be equal.

\medskip
We now use group-theoretic arguments to prove that $\Aut(A)=PGL(2,11)$. By \cite[Corollary 3.4]{aut} , we have $\hat{G}:=\Aut(A)$ is finite. Our approach is similar to the previous cases we handled. Let $Q$ be a minimal normal subgroup of $\hat G$. To begin with, we show that $\hat{G}$ has trivial soluble radical. For a contradiction, assume  that $Q$ is elementary abelian $p$-group for some prime $p$.  Then $Q$ can be regarded as a vector space over $\mathbb{F}_p$,
as well as a $\hat{G}$-module. Recall that $G_0=L_2(11)$ is of index $2$ in $G$, and that $G_0\geq E$. Set $R$ to be a minimal nontrivial subgroup of $Q$ invariant under $G_0$. Then $R$ is an irreducible $G_0$-module. 

\begin{lemma}\label{cent_E_in_R}
We have $C_R(E)=1$.
\end{lemma}

\begin{proof}
Since $\hat N=N_{\hat G}(E)=N\cong S_4$, we have that $E$ is self-centralised. If $C_R(E)$ were nontrivial, then $1\neq E\cap R\leq G_0\cap Q$, a contradiction, since $G_0\cap Q\unlhd G_0$
and $G_0$ is non-abelian simple.
\end{proof}

This immediately gives:

\begin{corollary}
We have $p\neq 2$. 
\end{corollary}

\begin{proof}
If $p=2$, then both $R$ and $E$ are $2$-groups, and so $C_R(E)\neq 1$, a contradiction.  
\end{proof}

Recall that Miyamoto involutions in axial algebras of Monster type $\left(\frac{1}{4},\frac{1}{32}\right)$ form a class of $6$-transpositions. This allows us to restrict $p$ even more.

\begin{lemma}
The prime $p$ lies in $\{3,5\}$.
\end{lemma} 

\begin{proof}
By contradiction, suppose that $p>7$. Let $1\neq e\in E$. By the choice of $E$, $e=\tau_i=\tau_{a_i}$ for some $i\in \{1,2,3\}$. It follows that $e$ belongs to a $6$-transposition class. In particular, it cannot invert an element of $R$ since $p\geq 7$.  Hence $e$ acts trivially on $R$ and this holds for all $e\in E$, a contradiction with Lemma \ref{cent_E_in_R}.
\end{proof}

We examine the cases $p=3$ and $p=5$ separately, and obtain a contradiction in each case using the known Brauer character tables of $L_2(11)$, for example found in GAP \cite{GAP}. The following observation applies in both cases.

\begin{lemma}\label{dim_divisible_3}
The dimension $n$ of $R$ is a multiple of $3$. That is, $n=3k$ for some $k\in \mathbb{N}$. In addition, the eigenvalues of each $1\neq e\in E$ are $1$, with multiplicity $k$, and $-1$, with multiplicity $2k$.
\end{lemma}

This is essentially \cite[Lemma 10.12]{aut}, so we just sketch the argument. Since $E$ is abelian, $R$ admits a basis with respect to which all $e\in E$ act diagonally. By Lemma \ref{cent_E_in_R}, two elements of $E$ invert every element of this basis, while the last one centralises it. Now the claim follows since the three nontrivial elements of $E$ are conjugate and hence all centralise the same number $k$ of basis elements. 

\begin{lemma}
We have $p\neq 3$.
\end{lemma}

\begin{proof}
By Lemma \ref{dim_divisible_3}, the dimension of $R$ is a multiple of $3$. The irreducible $3$-modular characters of $G_0$ are of dimensions $1$, $10$, and $12$. It follows that the four modules of dimension $12$ are the candidates. However, the irreducible characters of this dimension assume the value $0$ on the class of involutions, which mean equal multiplicities of $1$ and $-1$. It follows that none of them is suitable, so the result follows.
\end{proof}

We next show that $p$ cannot be equal to $5$.

\begin{lemma}
We have that $p\neq 5$.
\end{lemma}

\begin{proof}
In characteristic $5$, the character table of $L_2(11)$ has irreducible characters of degrees $1$, $10$ and $11$. But by Lemma \ref{dim_divisible_3}, $3|\dim(R)$, so there is no suitable module.
\end{proof}

Since we have ruled out all possible values of the prime $p$, we have shown that our assumption that the minimal normal subgroup $Q$ was abelian cannot hold. 

\begin{corollary}
The minimal normal subgroup $Q$ of $\hat{G}$ is non-abelian.
\end{corollary}

Consequently, $Q\cong L\times L\times \ldots \times L$ for a non-abelian simple group $L$. We aim to show that $Q=L$ is simple.

\begin{lemma}\label{S_is_Sylow_aut}
The Sylow $2$-subgroup $S\cong D_8\geq E$ of $N$ is a Sylow $2$-subgroup of $\hat{G}$.
\end{lemma}

\begin{proof}
Suppose that $\hat{S}$ is a Sylow $2$-subgroup of $\hat{G}$ containing $S$. If $\hat{S}$ properly contains $S$, then also so does $N_{\hat{S}}(S)$. Select $t\in N_{\hat{S}}(S)$. By Computation \ref{prelim_comp_L2_11} (c), the involutions in $S\setminus E$ do not correspond to axes.
Consequently, the involutions in $E$ cannot be conjugate to the involutions in $S\setminus E$. We conclude that $t$ must normalise $E$, i.e., $t\in \hat N=N\cong S_4$, a contradiction, since $S$ is self-normalised in $N$. 
\end{proof} 

It is well known (and apparently attributed to Burnside) that a non-abelian finite simple group cannot have Sylow $2$-subgroups of order $2$.

\begin{corollary}\label{Q_simple}
The minimal normal subgroup $Q$ of $\hat{G}$ is simple and $G_0\leq Q$.
\end{corollary}

\begin{proof}
By Lemma \ref{S_is_Sylow_aut}, $2^3$ is the $2$-part of $|\hat{G}|$, hence $Q$ must indeed be simple.
\end{proof}

We are now in a position to prove the ultimate result of this section.

\begin{proposition}
We have  $\Aut(A)=G\cong PGL(2, 11)$
\end{proposition}

\begin{proof}
Clearly, $Q$ is the unique minimal normal subgroup of $G$.
We have seen that the involutions in $S\setminus E$ do not correspond to any axes, and so they are not conjugate in $\hat{G}$ to the involutions in $E$. Thompson's Transfer Lemma \cite[Theorem 12.1.1]{KS} implies
that $\hat{G}$ has a subgroup $\hat{G}_0$ of index $2$ containing $E$ but not $S$. Since $E\subseteq \hat{G}_0 \cap Q$, we must have 
$Q\subseteq \hat{G}_0$. Thus, $E$ is a Sylow $2$-subgroup of $Q$. That is to say, $Q$ is a non-abelian simple group with
an elementary abelian Sylow $2$-subgroup of order $4$. By Walter's characterisation of finite groups with abelian Sylow
$2$-subgroups \cite{W}, $Q\cong L_2(q)$ for $q\equiv 3,5 \mbox{ mod }8$, $q>3$. We use the $6$-transposition property
to bound $q$. Involutions in $L_2(q)$, for odd $q$, invert tori of sizes $\frac{q-1}{2}$ and $\frac{q+1}{2}$. The $6$-transposition property implies that $\frac{q+1}{2}\leq 6$, and hence $q\leq 11$. This leaves $q=11$ and $Q=G_0$ as the only possibility. Since $\hat{G}$ has a unique non-abelian minimal normal subgroup $Q$, we have $Q=F^\ast(G)$, the generalised Fitting sugbroup of $\hat{G}$.
Then $\hat{G}$ must be isomorphic to a subgroup of $\Aut(Q)\cong PGL(2,11)$, and so $\hat G=G\cong PGL(2,11)$.
\end{proof}

\section{The $286$-dimensional algebra for $M_{11}$}\label{M11 algebra}

The $286$-dimensional algebra $A=A_{286}$ for $G:=M_{11}$ was constructed initially 
by Seress \cite{Se}. We have it from the library built using the expansion 
algorithm \cite{ParAxlAlg,axialconstruction}. The axet $X$ in $A$ is a single 
$G$-orbit of $165$ axes, which are in a bijection with the involutions from 
$G$. The shape of $A$ is $\bullet$ (see Example \ref{286 shape}). Let us record the basic information concerning 
the adjoint action on $A$ of the axes from $X$.

\begin{computation}
For an axis $a\in X$, the $1$-, $0$-, $\frac{1}{4}$- and 
$\frac{1}{32}$-eigenspaces of $\ad_a$ have dimension $1$, $128$, $25$, and 
$132$, respectively.
\end{computation}

The algebra $A$ admits a positive definite Frobenius form.\footnote{It is likely unique up to scale, but our function cannot complete this case because of high dimension. Fortunately, we do not need this fact in our proof.} 

\begin{computation}
\begin{enumerate}
\item The algebra $A$ has no Jordan axes; and
\item The axes in the axet $X$ have no twins.
\end{enumerate}
\end{computation}

This suggests that we should not expect any new axes and so aim to prove
that $\Aut(A)=G$. Naturally, we identify $G$ with a subgroup of 
$\hat G:=\Aut(A)$.

Let $S$ be a Sylow $2$-subgroup of $G$. Then it is well-known that $S$ is 
isomorphic to the semidihedral group of order $16$. So $S$ contains a 
unique central involution $z$ and a further class $C$ of involution of 
size four. The five involutions in $\{z\}\cup C$ generate a subgroup $P\lhd S$ 
isomorphic to the dihedral group of order $8$. Within $P$, the class $C$ 
decomposes as a union of two conjugacy classes, both of length two. We note 
that $P$ is generated by any two involutions from $C$ that are not conjugate 
in $P$. 

Switching to axes, let $a$ be the axis corresponding to $z$ and let 
$\{a_1,a_2\}$ and $\{a_3,a_4\}$ be the axes corresponding to the remaining 
involutions from $P$, organised in $P$-orbits. Let $E$ be the subalgebra 
of $A$ generated, say, by $a_1$ and $a_3$. Then $E\cong 4B$ is of dimension 
$5$. It contains all five axes above and, in fact, $\{a,a_1,a_2,a_3,a_4\}$ 
is a basis of $E$. Furthermore, $V:=\la a,a_1,a_2\ra$ and $W:=\la a,a_3,a_4\ra$ 
are the two $2A$-subalgebras of $E$. Let $z_i:=\tau_{a_i}\in C$ for $1\leq 
i\leq 4$.

Let $\hat S$ be the normaliser of $P$ in $\hat G$. Then $S\leq\hat S$, and we 
aim to show that, in fact, equality holds. Clearly, the elements of $\hat S$ 
act on the set of two elementary abelian $2$-subgroups of $P$ of order $4$, 
$\la z_1,z_2\ra$ and $\la z_3,z_4\ra$. Let $\hat P$ be the normaliser in $\hat S$ 
of both of these subgroups. Then $\hat P\cap S=P$ and hence $\hat S=S\hat P$.
This implies that in order to show that $\hat S=S$ it suffices to show that 
$\hat P=P$.

Any $\psi\in\hat S$ fixes $z$ and hence also $a$, and so it leaves the eigenspaces 
of $\ad_a$ invariant. We start the analysis of the action of $\psi$ on $A$ by 
concentrating on the $128$-dimensional subalgebra $U:=A_0(a)$. In order to split 
this large subalgebra into smaller pieces, we will need additional idempotents.
Note that $\psi\in\hat S$ also leaves $E$ invariant and fixes the identity 
$e:=\frac{4}{5}(a_1+a_2+a_3+a_4)+\frac{3}{5}a$ of $E$. (See for instance 
\cite{NortSakIdemps}.) Consequently, $\psi$ also leaves the eigenspaces of $\ad_e$ 
invariant. 

We note that $u:=e-a\in U$ and hence $U$ is invariant under $\ad_u=\ad_e-\ad_a$. 
Since $\ad_a$ acts trivially on $U$, we conclude that $U$ is $\ad_e$-invariant and 
so we can decompose $U$ into smaller pieces with respect to $\ad_e$. We have the 
following facts established computationally.  

\begin{computation} \label{eigenvalues of u}
\begin{enumerate}
\item The restriction $\ad_e|_U=\ad_u|_U$ has eigenvalues $1$, $0$, $\frac{3}{10}$, 
$\frac{7}{20}$, $\frac{1}{20}$, and $\frac{1}{10}$ with multiplicities $2$, $26$, 
$14$, $12$, $50$, and $24$, respectively; and 
\item both the $1$- and $0$-eigenspaces of $\ad_e|_U$ are subalgebras 
of $U$.  
\end{enumerate}
\end{computation}

Clearly, all these eigenspaces are invariant under $\psi\in\hat S$. However, we can 
decompose these eigenspaces even further. Set $D:=U_1(u)$. Note that $v=e-1_V$ and 
$w=e-1_W$ are contained in $D$. Hence, $D$ is spanned by these two idempotents and 
so it is fully contained in $E$. Therefore, we can easily calculate in $D$ and 
establish that $vw=0$, which means that $D\cong\F^2$ and it contains exactly four 
idempotents: $0$, $v$, $w$, and $u=v+w$. We also record for completeness the lengths 
of these idempotents: $v$ and $w$ have length $\frac{7}{5}$ and $u$ has length 
$\frac{14}{5}$. Note that $\hat P$ stabilises both $V$ and $W$ (and hence also $v$ 
and $w$), while the elements from $\hat S\setminus\hat P$ switch $V$ and $W$ and 
also $v$ and $w$.

The eigenvalues and their multiplicities for $\ad_v|_U$ and $\ad_w|_U$ can be deduced 
from Computation \ref{eigenvalues of u}. The exact and more detailed statement is as 
follows.

\begin{computation} \label{eigenvalues of v and w}
\begin{enumerate}
\item The restrictions $\ad_v|_U$ and $\ad_w|_U$ have eigenvalues $1$, $0$, 
$\frac{3}{10}$, and $\frac{1}{20}$ with respective multiplicities $1$, $59$, $13$, 
and $55$;
\item $v$ and $w$ are primitive axes of $U$ of almost Monster type 
$\left(\frac{3}{10},\frac{1}{20}\right)$;
\item the involution $\tau_v$ coincides with $z_3|_U=z_4|_U$ and, symmetrically, 
$\tau_w=z_1|_U=z_2|_U$.
\end{enumerate}
\end{computation}

Note that $\tau_v$ and $\tau_w$ are only defined on $U$.

\medskip
We have already mentioned that the automorphisms from $\hat S\setminus\hat P$ switch 
$v$ and $w$, and so such a regular behaviour of these elements is not unexpected. 
Moreover, since $vw=0$ and since $v$ and $w$ obey an almost Monster type fusion law, 
which is Seress, we conclude that $v$ and $w$ associate on $U$.

Our first goal is to determine the action of $\psi\in\hat S$ on the $26$-dimensional 
subalgebra $J=U_0(e)=U_0(u)$. Ideally, we would like to find idempotents in $J$ 
and then see how $\hat S$ can act on the idempotents. However, the dimension 
$26$ is high and we cannot do this directly. Instead, we find a smaller subalgebra 
of $J$ that is invariant under the action of $\hat S$.

\begin{computation} \label{action of S}
\begin{enumerate}
\item The group $P$ acts on $J$ as identity while $S$ induces on $J$ a group of order 
$2$;
\item the subalgebra $J_0$ of $J$ consisting of all vectors fixed by $S$ has dimension 
$14$.
\end{enumerate}
\end{computation} 

The subalgebra $J_0$ has a much better dimension. However, the issue is: why is $J_0$ 
invariant under $\hat S$? To show this, we temporarily leave $U$.

\begin{computation} \label{splitting J}
\begin{enumerate}
\item The eigenspace $A_0(u)$ has dimension $30$;
\item setting $K$ to be the $4$-dimensional orthogonal complement to $J$ in 
$A_0(u)$, we have that the projections to $J$ of the squares of the elements 
from $K$ span $J_0$. 
\end{enumerate}
\end{computation}

Since $u$ is fixed by $\hat S$, this group leaves both $J$ and $A_0(u)$ invariant and, 
consequently, it also preserves $K$. Applying Computation \ref{splitting J} (b), we now 
deduce that $\hat S$ also preserves $J_0$. 

Although we do not need it for our calculation, we note that $K$ contains the axis $a$ 
and it decomposes as $K=\la a\ra\perp K_0$, where the $3$-dimensional subspace $K_0$ 
is also invariant under $\hat S$. One can also check that $K_0\subseteq 
A_{\frac{1}{4}}(a)$, which implies that, when we compute the projections to $J$ of the 
squares of elements from $K$, we can restrict ourselves to $K_0$, as all products 
involving multiples of $a$ will have zero projection to $J$. Additionally, some of the elements 
from $S\setminus P$ negate all of $K_0$, while the rest fix $K_0$, which implies that squares of elements from 
$K_0$ are fixed by $S$ and so their projections to $J$ will certainly lie in $J_0$. 
Hence, the nontrivial part of the above calculation is just that they span all of 
$J_0$.

\medskip
Thus, we have established the following.

\begin{lemma}
The subalgebra $J_0$ is invariant under the action of $\hat S$.
\end{lemma}

This subalgebra is still too big to try and find all idempotents in it, but we can find 
idempotents of certain fixed length.

\begin{computation} \label{unique idempotents}
\begin{enumerate}
\item The subalgebra $J_0$ has a unique Frobenius form up to scale;
\item it has a unique idempotent of length $\frac{26}{5}$;
\item it also contains a unique idempotent of length $\frac{36}{5}$; and
\item the above two idempotents generate all of $J_0$.
\end{enumerate}
\end{computation}

Clearly, all elements of $\hat S$ fix both of these idempotents, which gives us the 
following important fact.

\begin{corollary} \label{trivial}
The group $\hat S$ acts trivially on $J_0$.
\end{corollary}

Let $B$ be the $12$-dimensional complement to $J_0$ in $J$. It is a module for $J_0$.

\begin{computation} \label{action on B}
\begin{enumerate}
\item The space of extensions of the identity automorphism of $J_0$ to $B$ is 
$1$-dimensional consisting of scalar actions on $B$;
\item there exist elements $b\in B$ such that the projection of $b^2$ to $J_0$ is 
nonzero. 
\end{enumerate}
\end{computation}

Note that the claim in (b) is equivalent to the fact that $B$ is not a subalgebra.

\medskip
We can now determine the action of $\hat S$ on the $J$.

\begin{lemma}
The group $\hat S$ induces on $J$ a group of order $2$ that fixes $J_0$ element-wise
and negates all elements of the complement $B$. 
\end{lemma}

\begin{proof}
For $\psi\in\hat S$, we have from  Corollary \ref{trivial} that $\psi$ acts as identity 
on $J_0$. Computation \ref{action on B} (a) now gives us that $\psi$ acts on $B$ as 
a scalar $\lm$. By Computation \ref{action on B} (b), we have $b\in B$ such that $b^2$ 
has a nonzero projection to $J_0$. Clearly, $(b^\psi)^2=\lm^2b^2$ has the same projection 
on $J_0$ as $b^2$, which implies that $\lm^2=1$. So we only have two possible actions 
for $\psi$. On the other hand, we know that the elements in $S\setminus P$ negate 
$B$, so the full group of order $2$ is induced on $J$.  
\end{proof}

We know that the kernel of $\hat S$ acting on $J$ has index $2$ in $\hat S$ and it contains 
$P$. A natural conjecture is that this kernel coincides with $\hat P$, and this is what we now aim to prove. Let $T:=U_{\frac{7}{20}}(e)$. According 
to Computation \ref{eigenvalues of u} (a), this is a subspace of $U$ of 
dimension $12$.
 
\begin{computation} \label{kern_Uc}
We have that $\lla T\rra=U$.
\end{computation}

Since $\ad_e|_U=\ad_u|_U$ and since $u=v+w$, with $v$ and $w$ associating, we 
can easily see that $T=T_1\oplus T_2$, where $T_1=U_{\frac{3}{10}}(v)\cap 
U_{\frac{1}{20}}(w)$ and $T_2=U_{\frac{1}{20}}(v)\cap U_{\frac{3}{10}}(w)$. 
Furthermore, both $T_1$ and $T_2$ have equal dimension $6$. Indeed, the 
elements from $S\setminus P$ swap $v$ and $w$ while preserving $T$, and so 
they swap $T_1$ and $T_2$.

Let $\rho$ be the non-identity automorphism induced by $\hat S$ on $J$. 
To prove that the kernel of $\hat S$ acting on $J$ is 
$\hat P$ we will show that every automorphism of $U$ acting on $J$ as $\rho$ 
necessarily switches $T_1$ and $T_2$. By contradiction, suppose that we have 
an automorphism acting on $J$ as $\rho$ and preserving both $T_1$ and $T_2$.
Recall that $B$ is the orthogonal complement of $J_0$ in $J$ and that $\rho$ 
negates $B$. 

\begin{computation}\label{extensions of rho}
\begin{enumerate}
\item The map $\rho$ admits a $1$-dimensional space of extensions to each 
$T_i$, $i=1,2$;
\item there are nonzero elements $t_i\in T_i$, $i=1,2$, and $b\in J_0$ such 
that $(t_i^2,b)\neq 0$.
\end{enumerate}
\end{computation}

Part (b) allows us to limit the possibilities for an automorphism $\phi$ 
extending $\rho$ and preserving both $T_1$ and $T_2$.

\begin{lemma}
There are four possible extensions $\phi$ of $\rho$ to $T_1$ and $T_2$.  
\end{lemma}

\begin{proof}
Let $\phi_i$, $i=1,2$, be the basis of the space of extensions of $\rho$ to 
$T_i$ found in Computation \ref{extensions of rho} (a). Then $\phi|_{T_i}= 
\lm_i\phi_i$, with the value of $\lm_i$ to be determined. Since $\phi$ is an automorphism of $U$, $((t_i)^2,b)=((t_i^\phi)^2,b^\phi)
=((t_i^{\lm_i\phi_i})^2,b^\rho)=((\lm_i t_i^{\phi_i})^2,b)
=\lm_i^2((t_i^\phi)^2,b)$, since $\rho$ is identity on $J_0$, and consequently, we get the equation
\begin{equation}\label{rho on J0}
\lm_i^2=\frac{(t_i^2,b)}{((t_i^{\phi_i})^2,b)}.
\end{equation}
The right side can be computed explicitly and it turns out to be a square 
$c_i^2$ for some $c_i\in\Q$. Hence this equation gives us that $\lm_i=\pm 
c_i$, and so indeed we have four possible extensions $\phi$ to $T=T_1\oplus 
T_2$ corresponding to $(\lm_1,\lm_2)\in\{(c_1,c_2),(c_1,-c_2),(-c_1,c_2),
(-c_1,-c_2)\}$. 
\end{proof}

Note that, for each such $\phi$, we know its action on $J$ and on $T$. By 
Computation \ref{kern_Uc}, $\la J\oplus T\ra=U$, so we can try extending these four 
maps to the entire $U$ and see if they are indeed automorphisms. As it turns 
out, none of them is.

\begin{computation}\label{extensions rho not auts}
None of the possible $(\lm_1,\lm_2)$ leads to an automorphism of $J$.
\end{computation}

This establishes the result we aimed to prove.

\begin{lemma}
The kernel of $\hat S$ acting on $J$ is precisely $\hat P$. 
\end{lemma}

\begin{proof}
Suppose that this is not true. Then $\hat P$ is not contained in the kernel, 
and hence there must exist $\phi\in\hat P$ that acts as no-identity on $J$, 
i.e., it induces $\rho$ on $J$. On the other hand, being in $\hat P$, $\phi$ 
fixes $v$ and $w$ and, consequently, it preserves $T_1$ and $T_2$. However, 
we have shown that no such $\phi$ exists, and the contradiction proves the 
claim.
\end{proof}

We are now in a position to show that $\hat P=P$ and hence $\hat S=S$. We 
first study the automorphisms $\hat P$ induces on $U$.  

\begin{computation}\label{kern_U}
The identity automorphism on $J$ has a $1$-dimensional space of extensions 
to each $T_i$, $i=1,2$.
\end{computation}

Clearly, this $1$-dimensional space of extensions is spanned by the 
identity map on $T_i$. This leads to the following. 

\begin{lemma} \label{action on U}
The automorphism group $\hat P$ induces on $U$ a group of order $4$, same 
as induced by $P=\la z_1,z_3\ra$.
\end{lemma}

\begin{proof}
Let $\psi\in\hat P$. By Computation \ref{kern_U}, we have $\psi|_{T_i}=\lm_i{\rm id}_{T_i}$ for some $\lm_i\in\F$, $i=1,2$. By Computation \ref{extensions of rho} (b), we can pick $t_i\in T_i$ such 
that $(t_i^2,b)\neq 0$ for some $b\in J$. Then 
$$
(t_i^2,b)=((t_i^2)^\psi,b^\psi)=((t_i^\psi)^2,b^\psi)=((\lm_i t_i)^2,b)
=\lm_i^2(t_i^2,b).
$$
Thus, $\lm_i^2=1$ and so $\lm_i=\pm 1$. It follows that  
$(\lm_1,\lm_2)$ is in 
$$\{(1,1),(1,-1),(-1,1),(-1,-1)\},$$
and so $\hat P$ induces on $U$ the group of order at most $4$, since $\la 
T\ra=U$. We note that the cases $(\lm_1,\lm_2)=(1,-1)$ and $(-1,1)$ correspond 
to the restrictions to $U$ of $\tau_{z_1}$ and $\tau_{z_3}$. Hence $P$ induces 
on $U$ the group of order exactly $4$, and consequently, so also does $\hat 
P$.   
\end{proof}

Note that $z_1$ can be substituted here with $z_2$ and, symmetrically, $z_3$ 
can be substituted with $z_4$.

\medskip
We next concentrate on $K:=A_{\frac{1}{32}}(a)$, of dimension $132$. We begin 
by determining the automorphisms from $\hat P$ acting as identity 
on $U$. 

\begin{computation}
The following hold:
\begin{enumerate}
\item $\lla K\rra=A$;
\item the identity automorphism $\rm{id}_U$ extends to a $1$-dimensional space of 
maps on $K$;
\item there exist $k\in K$ and $u_0\in U$ such that $(k^2,u_0)\neq 0$.
\end{enumerate}  
\end{computation}

From this we deduce the following. 

\begin{lemma} \label{kernel on U}
The kernel of $\hat P$ acting on $U$ coincides with $\la z\ra$.  
\end{lemma}

\begin{proof}
Part (b) of the preceding computation shows that an automorphism $\phi\in\hat P$ which 
restricts to $U$ as the identity, acts as a scalar $\lm\in\F$ on $K$. Taking $k$ and $u_0$ as in 
part (c), we get that $\lm^2(k^2,u_0)=((\lm k)^2,u_0)=((k^\phi)^2,u_0^\phi)=((k^2)^\phi,u_0^\phi)
=(k^2,u_0)\neq 0$. This implies $\lm^2=1$, that is, $\lm=\pm 1$. Since $A=\lla K\rra$ by part (a), 
the action of $\phi$ on $A$ is fully determined by its action on $K$. If $\lm=1$ then, clearly, $\phi$ 
is the identity automorphism. If $\lm=-1$ then $\phi=\tau_a=z$. Hence the kernel of $\hat P$ acting 
on $U$ is $\la z\ra$, as claimed. 
\end{proof}

Finally, we can identify our groups $\hat P$ and $\hat S$.

\begin{corollary}
We have that $\hat P=P$ and $\hat S=S$.
\end{corollary}

\begin{proof}
By Lemmas \ref{action on U} and \ref{kernel on U}, $|\hat P|=8=|P|$, which means that $\hat P=P$. 
Since $\hat P$ has index $2$ in $\hat S$, we also have that $|\hat S|=16=|S|$, so we also have that 
$\hat S=S$.
\end{proof}

We will also need the following.

\begin{corollary}
The subgroup $P$ is self-centralised in $\hat G=\Aut(A)$, i.e., $C_{\hat G}(P)=Z(P)=\la z\ra$. 
Furthermore, $S$ is self-normalised in $\hat G$, and so $S$ is a Sylow $2$-subgroup of $\hat G$.
\end{corollary}

\begin{proof}
If $\phi$ centralises $P$ then it fixes the axes $a$ and $a_i$, $i=1,2,3,4$, since it centralises the 
corresponding Miyamoto involutions and we have a bijection between axes and their Miyamoto involutions. 
In particular, $\phi$ acts trivially on the entire subalgebra $E$. Therefore, $\phi$ fixes $v$ and $w$, 
which means that $\phi\in\hat P=P$.

For the second claim, $N_{\hat G}(S)$ normalises $P$, since $P$ is characteristic in $S$. (Indeed, $P$ 
is generated by all the involutions from $S$.) Hence $N_{\hat G}(S)\leq N_{\hat G}(P)=\hat S=S$. 
Recall that $\hat G=\Aut(A)$ is a finite group by \cite{aut}. If $S$ were a proper subgroup of some 
Sylow $2$-subgroup $T$ of $\hat G$ then we would have that $N_T(S)>S$, which is not the case by the above. 
So $S$ is indeed a Sylow $2$-subgroup of $\hat G$.
\end{proof}

We are now in a position to establish the main result of this section using group-theoretic arguments. 
We begin by showing that $\hat{G}=\Aut(A)$ has no abelian normal subgroups. Suppose to the contrary, that 
$Q$ is an abelian minimal normal subgroup of $\hat{G}$. Then it must be an elementary abelian $p$-subgroup 
for some prime $p$. Consequently, $Q$ can be viewed as a vector space over $\F_p$, and a $\hat{G}$-module. 
Let $R$ be a minimal nontrivial subgroup of $Q$ invariant under $G$. Then $R$ is an irreducible $G$-module. 

\begin{lemma}\label{action_P0_nontrivial}
We have that $C_R(P)=1$.
\end{lemma}

\begin{proof}
If $C_R(P)\neq 1$ then $P\cap R\neq 1$, since $P$ is self-centralised. Since $R\leq Q$ and $R\leq G$, 
we conclude that $G\cap Q\neq 1$, which is a contradiction, since $G\cong M_{11}$ is non-abelian simple.
\end{proof}

In particular, this statement implies the following.

\begin{corollary}
The prime $p$ is not $2$.
\end{corollary}

\begin{proof}
Suppose that $p=2$. Then $C_R(P)$ cannot be trivial as both $R$ and $P$ are $2$-groups. 
\end{proof}

We conclude that $p$ is an odd prime. We will now utilise the fact that Miyamoto involutions form a class 
of $6$-transpositions, and this will allow us to restrict the value of $p$ even further. 

\begin{lemma}
We have that $p=3$ or $5$.
\end{lemma}

\begin{proof}
By Lemma \ref{action_P0_nontrivial}, $G$ acts on $R$ nontrivially, and hence $z$ acts on $R$ nontrivially. 
Because $p\ne 2$, there must be an element $1\ne r\in R$ that is inverted by $z$. Then we calculate that 
$|zz^r|=p\leq 6$, since $z$ belongs to a $6$-transposition class. 
\end{proof}

We will now rule out the cases $p=3$ and $p=5$ in turn using the known irreducible modules of $M_{11}$ over 
the respective fields $\mathbb{F}_p$. Our strategy is to form the semidirect product $V\rtimes G$ and check 
whether it is a $6$-transposition group. If not, then we can be sure that $R$ is not isomorphic to $V$ as 
a $G$-module. Indeed, $RG\cong V\rtimes G$, and $RG$ is clearly a $6$-transposition group. If $V\rtimes G$ 
happens to be a $6$-transposition group then we use additional arguments.

We first note the following observation which applies for both possible values of $p$.

\begin{lemma}
Let $p=3$ or $5$, $n>1$, and $V$ be an $n$-dimensional irreducible $G$-module over $\mathbb{F}_p$. Then
$V\rtimes G\cong p^n:M_{11}$ has a unique class of involutions. 
\end{lemma}

This is because $G$ contains a Sylow $2$-subgroup of $V\rtimes G$ and $G$ has a single class of involutions. 

\begin{lemma} \label{not 3}
We have that $p\ne 3$.
\end{lemma}

\begin{proof}
The $3$-modular Brauer character table for $G$ is available in MAGMA \cite{MAGMA}. The nontrivial modules 
$V$ are of dimensions $5$ (two such), $10$ (three such), $24$ and $45$. They all can be constructed within 
MAGMA, as can also be the semidirect products $V\rtimes G$. First, let $\dim V=5$. For both irreducible 
modules of this dimension, the unique class of $3^2\times 165=1485$ involutions is a $6$-transposition class 
for $V\rtimes G$. Assuming that $RG\cong V\rtimes G$, this gives us a set $\hat X$ of $1485$ axes in $A$.
Furthermore, the possible shapes on $\hat X$ are quite limited, as the shape on the original set
$X\subset\hat X$ of the $165$ axes corresponding to $G=M_{11}$ is known. Furthermore, once the
shape is known, we can write the Gram matrix of the set $\hat X$ with respect to the Frobenius form.
This is because the value of $(c,d)$ for axes $c$ and $d$ is defined by the type of the $2$-generated
algebra $\lla c,d\rra$ recorded in the shape.

As it turns out, for one of the $5$-dimensional modules $V$ and three of the four possible shapes,
the Gram matrix has full rank $1485$, while the Gram matrix for the fourth shape has rank $1430$.
This is clearly a contradiction, as $\dim A=286$. For the other $5$-dimensional module $V$,
we have four possible shapes and the rank of the Gram matrix in the corresponding cases
is $1265$ and $1386$, and two cases have full rank $1485$, still a contradiction.

For all other irreducible modules $V$, the group $V\rtimes G$ is not a $6$-transposition group, which rules them out. Thus, $p\neq 3$.
\end{proof}

\begin{lemma} \label {not 5}
We have $p\ne 5$. 
\end{lemma}

\begin{proof}
The nontrivial irreducible modules for $G$ in characteristic $5$ have dimensions $10$, $11$, $16$ (two such), $20$, $45$ and $55$. For all these modules,
$\tilde{P}= V\rtimes G\cong 5^n: M_{11}$ is not a $6$-transposition group, a contradiction. This rules out $p=5$. 
\end{proof}

The details of the computations from Lemmas \ref{not 3} and \ref{not 5} can 
be found on GitHub \cite{m11 aut}. 

\medskip
We conclude that every minimal normal subgroup $Q$ of $\hat{G}$ is non-abelian, and hence is a direct product of a finite number of copies of a non-abelian
simple group $L$, i.e., $Q\cong L\times L\ldots \times L$. We aim to show that $Q\cong L$ is simple. 

\begin{lemma}\label{min_normal_contains_G}
Every minimal normal subgroup $Q$ of $\hat{G}$ is a non-abelian simple group. Furthermore, $G\leq Q$.
\end{lemma}

\begin{proof}
Recall that $S$ is a Sylow $2$-subgroup of $\hat G$ and $|S|=2^4$. Assuming that $Q\cong L\times L\cdots\times L$, with the number of factors $k$, and using that the Sylow $2$-subgroup of $L$ cannot be cyclic, we immediately see that $k\leq 2$. Furthermore, if $k=2$ then the Sylow
$2$-subgroup of $L$ is of order $2^2$ and hence abelian. This means that the Sylow $2$-subgroup of $Q\cong L\times L$ is abelian of
order $2^4$, which is a contradiction, since $S$ is not abelian.

Thus, $k=1$. Since $S$ is Sylow in $\hat G$, we have that $S\cap Q\neq 1$, and consequently, $G\cap Q\neq 1$. Since $G$ is simple, we conclude that $G\leq Q$.
\end{proof}

We are now in a position to prove the ultimate result of this section.

\begin{proposition}
The full automorphism group $\hat{G}$ of $A$ coincides with $G=M_{11}$.
\end{proposition}

\begin{proof}
Taking a minimal normal subgroup $Q$ of $\hat G$, we have by Lemma \ref{min_normal_contains_G} that $Q$ is simple and $G\leq Q$. Since $S\leq G$ is Sylow in $\hat G$, we have that $S$ is Sylow in $Q$, and so $Q$ is a finite simple group with a semi-dihedral Sylow subgroup of order $2^4$. By \cite{Alperin}, $Q$ is one of $M_{11}$, $L_3(q), q\equiv -1 \mbox{ mod }4$,
or $U_3(q), q\equiv 1 \mbox{ mod } 4$. Maximal subgroups of $L_3(q)$, $q$ odd, are known (see for instance \cite{Bloom}),
and looking through the list, we easily see that $M_{11}$ cannot be a subgroup of $L_3(q)$. Similarly,
$M_{11}$ is not embeddable in $U_3(q)$, for any $q$, and we conclude that $Q=M_{11}$. 

Thus, $G$ is a minimal normal subgroup of $\hat G$ and clearly this means that $G$ is the unique minimal normal subgroup of $\hat G$. This implies that $\hat G$ is isomorphic to a subgroup of $\Aut(M_{11})\cong M_{11}$, and this yields the result. 
\end{proof}

\end{document}